\pdfoutput=1
\documentclass[11pt]{amsart}
\usepackage{amsmath}
\usepackage{amssymb}
\usepackage{hyperref} %added
\hypersetup{colorlinks=true, linkcolor=blue, citecolor = blue, urlcolor=blue, filecolor=magenta}
\usepackage{enumerate}
\usepackage{slashed}
\usepackage{caption}
\usepackage{subcaption}
\usepackage{newlfont}
\usepackage{amsrefs}
\usepackage{comment}
\usepackage[normalem]{ulem}
\usepackage{accents}
\usepackage{leftidx}

\usepackage{harpoon}
\usepackage[pdftex]{graphicx}
\usepackage{esint}
\usepackage{relsize}
\usepackage[latin1]{inputenc}
\usepackage{xcolor}

\definecolor{brass}{rgb}{0.71, 0.65, 0.36}

\usepackage{tikz-cd}
\theoremstyle{plain}
\newtheorem{theorem}{Theorem}[section]
\newtheorem{lemma}[theorem]{Lemma}

\newtheorem{corollary}[theorem]{Corollary}

\newtheorem{proposition}[theorem]{Proposition}

\theoremstyle{definition}
\newtheorem{remark}[theorem]{Remark}

\numberwithin{equation}{section}
\def\N{\mathbb{N}}

\def\R{\mathbb{R}}

\DeclareMathOperator{\Vol}{Vol}

\newcommand{\capacity}{\mathrm{Cap}}

\def\sup{\operatorname{sup}}
\def\inf{\operatorname{inf}}

\def\Vol{\operatorname{Vol}}

\def\sup{\operatorname{sup}}

\theoremstyle{plain}

\numberwithin{equation}{section}

\allowdisplaybreaks

\begin{document}
\title[The Spacetime Penrose Inequality With Suboptimal Constant]{Proof of the Spacetime Penrose Inequality With Suboptimal Constant in the Asymptotically Flat and Asymptotically Hyperboloidal Regimes}
%\title[Penrose Inequalities for General Initial Data]{A Spacetime Penrose Inequality for Asymptotically Flat and Asymptotically %Hyperboloidal Initial Data Via Harmonic Level Sets}

\author[Allen]{Brian Allen}
\address{Department of Mathematics\\
Lehman College, City University of New York\\
Bronx, NY 10468, USA}
\email{brian.allen1@lehman.cuny.edu}

\author[Bryden]{Edward Bryden}
\address{Department of Mathematics\\
University of Antwerp\\
Antwerpen, BEL}
\email{edward.bryden@uantwerpen.be}

\author[Kazaras]{Demetre Kazaras}
\address{Department of Mathematics\\
Michigan State University\\
East Lansing, MI 48824, USA}
\email{kazarasd@msu.edu}

\author[Khuri]{Marcus Khuri}
\address{Department of Mathematics\\
Stony Brook University\\
Stony Brook, NY 11794, USA}
\email{marcus.khuri@stonybrook.edu}

%\date{\today}

\thanks{M. Khuri acknowledges the support of NSF Grant DMS-2405045.}

\begin{abstract}
We establish mass lower bounds of Penrose-type in the setting of $3$-dimensional initial data sets for the Einstein equations satisfying the dominant energy condition, which are either asymptotically flat or asymptotically hyperboloidal. More precisely, the lower bound consists of a universal constant multiplied by the square root of the minimal area required to enclose the outermost apparent horizon. Here the outermost apparent horizon may contain both marginally outer trapped (MOTS) and marginally inner trapped (MITS) components. The proof is based on the harmonic level set approach to the positive mass theorem, combined with the Jang equation and techniques arising from the stability argument of Dong-Song \cite{Dong-Song}. As a corollary, we also obtain a version of the Penrose inequality for 3-dimensional asymptotically hyperbolic Riemannian manifolds.
\end{abstract}

\maketitle

\section{Introduction}
\label{sec1} \setcounter{equation}{0}
\setcounter{section}{1}

Motivated by the weak cosmic censorship conjecture \cite{Penrose}, Penrose proposed an inequality \cite{Penrose1} relating the ADM mass (or energy) $m$ of a 4-dimensional asymptotically flat spacetime to the cross-sectional area $\mathcal{A}_{e}$ of the event horizons it contains, which takes the form
\begin{equation}
m\geq\sqrt{\frac{\mathcal{A}_e}{16\pi}}.
\end{equation}
This inequality is typically reformulated in the setting of initial data sets. Consider a triple $(M,g,k)$ consisting of a $3$-dimensional connected and orientable manifold $M$, a complete Riemannian metric $g$, along with a symmetric 2-tensor $k$ representing the second fundamental form of an embedding into spacetime.  These quantities are assumed to be smooth and satisfy the constraint equations
\begin{equation}\label{densityem}
16 \pi \mu=R_g+(\text{Tr}_g k)^2-|k|_g^2, \qquad\text{ }
8\pi J=\text{div}_g\left(k-(\text{Tr}_g k)g\right),
\end{equation}
where $R_g$ denotes scalar curvature, and $\mu$, $J$ are the energy-momentum density of matter fields. 
We will say that the \textit{dominant energy condition} holds if $\mu\ge |J|_g$. Furthermore, the data will be called \textit{asymptotically flat} if there exists a compact set $\mathcal{K}$ such that $M\setminus \mathcal{K}=\cup_{\ell=1}^{\ell_0}M^{\ell}_{end}$, where the ends $M_{end}^{\ell}$ are pairwise disjoint and each is diffeomorphic to the complement of a Euclidean ball $\mathbb{R}^3 \setminus B_1$. Moreover, if $\varphi$ denotes the diffeomorphism from  Euclidean space of coordinates $x$ to an end then
%there is a diffeomorphism $\varphi: M \setminus \mathcal{C}\rightarrow\mathbb{R}^3\setminus B_1$, such that in %the Caretesian coordinates $x$ associated with this map
\begin{align}\label{1}
\begin{split}
\varphi^* g=\delta+O_2(|x|^{-q}), \quad\quad\quad
\varphi^* k= O_1(|x|^{-q-1}),\\
\varphi^*\mu, \varphi^*J=O(|x|^{-2q-2}),\quad\quad\quad \varphi^*\mathrm{Tr}_g k=O(|x|^{-2q-1}),
\end{split}
\end{align}
for some $q>\tfrac{1}{2}$ where $O_l(|x|^{-q})$ represents a tensor in the weighted space $C^l_{-q}(\mathbb{R}^3)$. The additional decay on the trace of $k$ is usually not included in the definition of asymptotically flatness, but is recorded here for usage with the Jang equation. Under these asymptotics the ADM energy and linear momentum of each end are well-defined \cites{Bartnik,Chrusciel} and given by
\begin{equation}\label{fpajfpojapojfpoqhhh}
E\!=\!\lim_{r\rightarrow\infty}\frac{1}{16\pi}
\int_{S_{r}}\sum_i (g_{ij,i}-g_{ii,j})\nu^{j}dA,\quad
P_i \!=\!\lim_{r\rightarrow\infty}\frac{1}{8\pi}
\int_{S_{r}}(k_{ij}-(\text{Tr}_g k)g_{ij})\nu^{j}dA,
\end{equation}
where $S_r$ are coordinate spheres with unit outer normal $\nu$ and area element $dA$.
The ADM mass is then the Lorentz length of the energy-momentum vector, $m=\sqrt{E^2-|P|^2}$.

The role of the event horizon is replaced by that of an apparent horizon, which may be computed directly from the initial data. Let $\Sigma$ be a closed 2-sided hypersurface in $M$ with null expansions $\theta_{\pm}= H \pm \mathrm{Tr}_{\Sigma} k$, where $H$ is the mean curvature of $\Sigma$ computed as the tangential divergence of the unit normal $\nu$ pointing towards a designated end $M_{end}$. The null expansions are themselves (spacetime) mean curvatures, namely in the null directions $\nu\pm n$ where $n$ is the future pointing timelike normal to the slice $(M,g,k)$. These quantities may be interpreted physically as determining the rate of change of area for a shell of light emanating from the surface in the outward future/past direction, and they can be used to measure the strength of the gravitational field. In particular, the gravitational field is strong near the surface $\Sigma$ if it is \textit{outer or inner trapped}, that is $\theta_+< 0$ or $\theta_- <0$. Moreover, $\Sigma$ is called a \textit{marginally outer or inner trapped surface} (MOTS or MITS) if $\theta_+ = 0$ or $\theta_-=0$. These types of surfaces are also referred to as future or past apparent horizons, and naturally arise as boundaries of future or past trapped regions \cite{AnderssonMetzger}. Furthermore, a collection $\Sigma$ of disjoint MOTS and MITS components will be called an \textit{outermost apparent horizon} with respect to $M_{end}$, if $\Sigma$ is not enclosed from the perspective of $M_{end}$ by any other disjoint collection of apparent horizon components. The existence of an outermost apparent horizon for each end may be obtained from \cite{AnderssonMetzger}, by first finding the outermost MOTS and outermost MITS separately, and then removing components until all are disjoint. Although the outermost MOTS and outermost MITS are individually unique, the second step in which components are removed entails a choice, and thus the resulting outermost apparent horizon may not be unique.

The conjectured Penrose inequality for initial data sets may then be recast as
\begin{equation}\label{ajfoij}
m\geq \sqrt{\frac{\mathcal{A}}{16\pi}}
\end{equation}
whenever the dominant energy condition holds, where $\mathcal{A}$ is the minimum area required to enclose an outermost apparent horizon \cite{Mars}. Equality should be achieved only for slices of the Schwarzschild spacetime.
In the (Riemannian) time symmetric case when $k=0$, this inequality has been confirmed by Huisken-Ilmanen \cite{HI} and Agostiniani-Mantegazza-Mazzieri-Oronzio \cite{Mazzierietal} for a single component black hole via inverse mean curvature flow and $p$-harmonic functions respectively. The multiple black hole scenario was treated by Bray \cite{Bray} using a conformal flow, and this approach has been generalized by Bray-Lee \cite{Bray:2007opu} for dimensions up to 7. Moreover, the Penrose-like inequality of McCormick-Miao \cite{McCormickPengzi} applies for outerminimizing boundaries that are not necessarily minimal hypersurfaces. Within the context of the spacetime setting, there are few results outside of strong symmetry assumptions. In fact, for this regime the conjectured inequality has only been verified in the spherically symmetric case \cites{Hayward,MM} with the rigidity statement also obtained in \cites{Bray:2009au,BKS}, \cite{Leetext}*{Theorem 7.46}, and in the general cohomogeneity one setting by Khuri-Kunduri \cite{KhuriKunduri}; these results hold in higher dimensions as well. Furthermore, Penrose-like inequalities have been found in the spacetime setting by the fourth author \cites{Khuri,Khuri1} and Yu \cite{Yu}, inspired by the Riemannian result of Herzlich \cite{Herzlich}. These latter results replace $\mathcal{A}$ with the area of an outermost apparent horizon, and involve a suboptimal constant which depends in a crucial way on the initial data. In the present work we will show that there is a universal (suboptimal) constant for which the inequality \eqref{ajfoij} holds.

\begin{theorem}\label{t:main}
Let $(M,g,k)$ be a complete asymptotically flat initial data set satisfying the dominant energy condition. There exists a universal constant $\mathcal{C}<10^{18}$ such that
\begin{equation}%\label{aofhoaihfoiihh}
m\geq \sqrt{\frac{\mathcal{A}}{\mathcal{C}}},
\end{equation}
where $m$ is the ADM mass of an end and $\mathcal{A}$ is the minimal area required to enclose an outermost apparent horizon associated with the end. 
\end{theorem}

\begin{remark}
We point out that the statement of this theorem applies to at least one outermost apparent horizon for any given end. It is likely that the methods of this paper can be suitably modified so that the statement holds for any outer most apparent horizon,
although this will not be pursued here.
\end{remark}

\begin{comment}
\begin{theorem}\label{t:main}
    There is a constant $C$ so that the following holds: Suppose $(M^3,g,k)$ is an asymptotically flat initial data set satisfying the dominant energy condition. Then the minimal area $A$ required to enclose an outermost apparent horizon satisfies 
    \begin{equation}
        \sqrt{A}\leq C\cdot E.
    \end{equation}
\end{theorem}

\begin{corollary}\label{c:maincor}
        There is a constant $C$ so that the following holds: Suppose $(M^3,g,k)$ is a complete asymptotically flat initial data set satisfying the vacuum constraint equations $\mu=|J|=0$. Then the minimal area $A$ required to enclose an outermost apparent horizon satisfies 
    \begin{equation}
        \sqrt{A}\leq C\cdot m.
    \end{equation}
\end{corollary}
\end{comment}

The heuristic physical arguments of Penrose that motivate \eqref{ajfoij} also apply to a certain class of asymptotically hyperbolic initial data modelled on slices of asymptotically flat spacetimes which asymptote to a light cone, and thus have a nontrivial intersection with null infinity. Let $(\mathbb{H}^3,b)$ denote the hyperboloidal model of 3-dimensional hyperbolic space arising from the hyperboloid $t=\sqrt{1+|x|^2}$ in Minkowski space. Thus, $\mathbb{H}^3$ is identified with $[0,\infty)\times S^2$ and $b=(1+r^2)^{-1}dr^2 +r^2 \sigma$, where $\sigma$ is the unit round metric on the 2-sphere. We will say that an initial data set is \textit{asymptotically hyperboloidal} if there exists a compact set $\mathcal{K}$ such that $M\setminus \mathcal{K}=\cup_{\ell=1}^{\ell_0}M^{\ell}_{end}$, where the ends $M_{end}^{\ell}$ are pairwise disjoint and each is diffeomorphic to the complement of a hyperbolic ball $\mathbb{H}^3 \setminus B_1$. Moreover, if $\varphi$ denotes the diffeomorphism from the hyperboloidal model to an asymptotic end then
%there is a diffeomorphism $\varphi: M \setminus \mathcal{C}\rightarrow\mathbb{R}^3\setminus B_1$, such that in %the Caretesian coordinates $x$ associated with this map
\begin{equation}\label{1aojijhiji}
\varphi^* g=b+\frac{\mathbf{m}}{r}+O_5(r^{-2}), \quad\quad
\varphi^* k=b+\frac{\mathbf{p}}{r}+O_4(r^{-2}), \quad\quad \varphi^* \mu, \varphi^* J =O_3(r^{-3-q_0}),
\end{equation}
for some $q_0 >0$ where $\mathbf{m}$ and $\mathbf{p}$ are symmetric 2-tensors on $S^2$ and $O_l(r^{-q})$ represents a tensor in the weighted space $C^l_{-q}(\mathbb{H}^3)$. 
%\textcolor{red}{[Check that the condition on $\mu,J$ does not hold automatically with Wang asymptotics.]} 
This decay is often referred to as \textit{Wang asymptotics} due to the initial study \cite{wang} of mass in the asymptotically hyperbolic setting. The assumption on the number of derivatives and the asymptotics themselves, are stronger requirements than what is traditionally needed for positive mass results. This is due to our reliance on the Jang equation, and in particular the analysis of this equation by Sakovich \cite{Sakovich} in this regime. The total energy is the well-defined \cites{CJL,Michel} and given by
\begin{equation}
E_{\mathrm{hyp}}=\frac{1}{16\pi}\int_{S^2}\mathrm{Tr}_{\sigma}(\mathbf{m}+2\mathbf{p})dA_{\sigma}.
\end{equation}

%The next result concerns asymptotically hyperbolic initial data, as described in Sakovich \cite{Sakovich}. Let $(\mathbb{H}^3,h)$ denote hyperbolic 3-space. We say that an end $\mathcal{E}$ of an initial data set $(M,g,k)$ is {\emph{asymptotically hyperbolic}} if there is $\tau>3/2$, $\beta>0$, and a diffeomorphism $\Phi:\mathcal{E}\to\mathbb{H}^3 \setminus B_1$, so that
%\begin{align}\label{2}
%\Phi_*g-h\in C^{3,\beta}_\tau(\mathbb{H}^3), \quad\quad
%\Phi_*(k-g)\in C^{2,\beta}_\tau(\mathbb{H}^3),\quad\quad
%\Phi_*\mu,\Phi_*J\in C^{1,\beta}_{\tau+\frac32}(\mathbb{H}^3),
%\end{align}
%where the weighted H{\"o}lder spaces are defined as in \cite{DahlSak}. Compared to standard definitions, we impose additional derivative control on $(g,k)$ and decay of $\mu,J$. This is required when solving the Jang equation as constructed by Sakovich \cite{Sakovich}.

It was shown by Neves \cite{Neves} that the naive approach of a modified Hawking mass monotonicity along inverse mean curvature flow is not sufficient to establish hyperbolic Penrose inequalities. Even in the umbilic case when $k=g$ and the dominant energy condition reduces to the scalar curvature lower bound $R_g\geq -6$, little is known about this inequality outside the spherically symmetric \cite{MHou}, cohomogeneity one \cite{KhuriKunduri}, and graphical \cite{GL} settings. There is a different version of the hyperbolic Penrose inequality \cite{BrayChrusciel}*{Section 4.1}, which is modelled on asymptotically AdS spaces. Although not motivated by the Penrose heuristic arguments, it has been proven when $k = 0$ for small perturbations of the Schwarzschild-AdS manifold by Ambrozio \cite{LA} (see also \cite{KhuriKopinski}), and for graphs by Dahl-Gicquaud-Sakovich \cite{DGS} and 
de Lima-Gir\~{a}o \cite{LG}. 
Related results include a Penrose-like inequality for asymptotically locally hyperbolic (ALH) manifolds by Alaee-Hung-Khuri \cite{AHK}, graphical ALH Penrose inequalities by Ge-Wang-Wu-Xia \cite{GWW} and de Lima-Gir\~{a}o \cite{LG}, nonpositive mass ALH inequalities of Penrose-type by Lee-Neves \cite{LeeNeves}, and a null Penrose inequality by Roesch \cite{Roesch}. Here we will show that there is a universal (suboptimal) constant for which the spacetime asymptotically hyperboloidal Penrose inequality holds. 

\begin{theorem}\label{t:hyp}
Let $(M,g,k)$ be a complete asymptotically hyperboloidal initial data set satisfying the dominant energy condition. There exists a universal constant $\mathcal{C}<10^{18}$ such that
\begin{equation}%\label{aofhoaihfoiihh}
E_{\mathrm{hyp}}\geq \sqrt{\frac{\mathcal{A}}{\mathcal{C}}},
\end{equation}
where $E_{\mathrm{hyp}}$ is the total energy of an end and $\mathcal{A}$ is the minimal area required to enclose an outermost apparent horizon associated with the end. 
\end{theorem}

We point out that, unlike the asymptotically flat case, this theorem is novel even in the Riemannian (umbilic) regime. In this situation the outermost minimal surface is replaced by an outerminimizing surface of mean curvature $H=2$, and the associated Penrose inequality was conjectured in \cite{wang}.

\begin{corollary}\label{c:hyp}
Let $(M,g)$ be a complete Riemannian 3-manifold which is asymptotically hyperbolic in the sense of \eqref{1aojijhiji}. Assume further that the scalar curvature satisfies $R_g\geq -6$, and let $A$ denote the area of the outermost surface of mean curvature $H=2$ with respect to a designated end. There exists a universal constant $\mathcal{C}<10^{18}$ such that
\begin{equation}%\label{aofhoaihfoiihh}
E_{\mathrm{hyp}}\geq \sqrt{\frac{A}{\mathcal{C}}},
\end{equation}
where $E_{\mathrm{hyp}}$ is the total energy of the designated end.
\end{corollary}

An outline of the main arguments, and the organization of this manuscript are as follows. In Section \ref{s:genext}, as preparation for the use of level set techniques, we construct a so called generalized exterior region which simplifies the topology while not decreasing the minimal area required to enclose an outermost apparent horizon and preserving the mass/energy of a designated end.  The Jang equation is then solved over the generalized exterior region with blow-up at an outermost apparent horizon. In Section \ref{s:massofjang} it is shown that the total energy of the designated end within the Jang graph bounds the second derivatives of a harmonic map $U$ from the Jang graph to $\mathbb{R}^3$. Next, following Dong-Song \cite{Dong-Song} we consider the set $\Omega_{0,\tau}$ on which $U$ is pointwise-close to a local isometry, and prove in Section \ref{s:DSinequality} that the area of $\partial \Omega_{0,\tau}$ is universally bounded above by a multiple of the total energy squared. The desired Penrose inequalities are then established in Section \ref{s:mainproof}, by exploiting the cylindrical geometry of the nondesignated ends within the Jang graph to show that a projection of $\partial \Omega_{0,\tau}$ encloses the outermost apparent horizon, so that its area is an upper bound for $\mathcal{A}$.

\begin{comment}
\textcolor{red}{The following is a reorganization of a portion of \cite{Dong-Song}, which we observe is a rather general statement about a collection of functions on a Riemannian manifold. Below and throughout, $|\Sigma|$ denotes the area of a smooth surface in a Riemannian manifold.
\begin{theorem}\cite[Section 3]{Dong-Song}\label{thm-intro Area Energy Inequality}
    There exists a $C<10^{9}$ so that the following holds: Suppose $(M^3,g)$ is a complete Riemannian manifold with functions $u^1,u^2,u^3\in C^\infty(M)$ and set 
    \begin{equation}
        F(x)=\sum_{i,j=1}^3\left(\langle \nabla u^i,\nabla u^j\rangle -\delta^{ij}\right)^2.
    \end{equation}
    If $F$ is nonconstant and proper on $F^{-1}((0,b))$ for some $b>0$, then there is a regular value $\tau_0 \in (0,b)$ of $F$ so that
        \begin{align}
            \int_{\{F<b\}}|\nabla F|^2dV\geq\sqrt{\frac{|F^{-1}(\tau_0)|}{C}}.
        \end{align}
\end{theorem}} 
\end{comment}

\subsection*{Acknowledgements} The authors would like to thank Hubert Bray for his interest in these results, Conghan Dong for insightful discussions concerning the work \cite{Dong-PMT_Stability}, and Pei-Ken Hung for %suggesting Theorem \ref{t:hyp}, 
suggestions concerning the asymptotically hyperboloidal case.

\section{The Generalized Exterior Region}
\label{s:genext}
\setcounter{equation}{0}
\setcounter{section}{2}

In this section we show how to reduce the topology of initial data sets, so that harmonic level set techniques will be applicable. In particular, this construction reduces the Penrose inequality for general initial data sets to the Penrose inequality for initial data having vanishing first Betti number. This is achieved by passing to a \textit{generalized exterior region} associated with a given end, as described in the following result. Such a construction was first described in \cite{HKK} for the asymptotically flat case. Here we extend it to the asymptotically hyperboloidal setting as well, and refine certain properties to show that the minimal area required to enclose the relevant apparent horizons does not decrease in the process.

\begin{lemma}\label{l:exteriorregion} Let $(M,g,k)$ be an initial data set which is either asymptotically flat or asymptotically hyperboloidal, and satisfies the dominant energy condition. Then for each end $M_{end}$ there exists an associated outermost apparent horizon $\Sigma\subset M$, and a new initial data set $(M_{\mathrm{ext}},g_{\mathrm{ext}},k_{\mathrm{ext}})$ with boundary satisfying the following properties.
%Suppose $(M,g,k)$ is a complete initial data set with a single end $\mathcal{E}$ which is either asymptotically flat as in %\eqref{1} or asymptotically hyperbolic as in \eqref{1aojijhiji}. If $(M,g,k)$ satisfies the DEC and $\Sigma\subset M$ is an %outermost apparent horizon, then there is a second initial data set $(M_{\mathrm{ext}},g_{\mathrm{ext}},k_{\mathrm{ext}})$ %inside of a covering space $\pi:\widehat{M}\to M$ so that
\begin{enumerate}
\item The second relative homology is trivial, $H_2(M_{\mathrm{ext}},\partial M_{\mathrm{ext}};\mathbb{Z})=0$.
\item $M_{\mathrm{ext}}$ has a single end which is isometric as initial data to $(M_{end},g,k)$.
\item The boundary $\partial M_{\mathrm{ext}}$ is an outermost apparent horizon.
\item The minimal area required to enclose $\partial M_{\mathrm{ext}}$ is not less than the minimal area required to enclose $\Sigma$.
\end{enumerate}
\end{lemma}

\begin{remark}
Loosely speaking, $M_{\mathrm{ext}}$ is meant to represent the region outside of the outermost apparent horizon $\Sigma$. However, since the generalized exterior region is obtained by passing to covering spaces, components of $\partial M_{\mathrm{ext}}$ may project to apparent horizons in $M$ which are merely immersed. 
\end{remark}

\begin{proof}
As mentioned above, the existence of a generalized exterior region satisfying $(1)$, $(2)$, and $(3)$ was proven in \cite{HKK}*{Proposition 2.1} for the asymptotically flat case. Here we will describe how to appropriately modify the construction in order to achieve $(4)$, and to establish that the procedure can be performed in the asymptotically hyperboloidal context as well.
%We review the generalized exterior region construction of \cite[Proposition 2.1]{HKK}, where an initial data set $M_{ext}$ lying %in some cover of $M$ was obtained, satisfying conditions 1, 2, and 3 in the asymptotically flat setting. In fact, this %construction can be carried out in the asymptotically hyperbolic case and $M_{ext}$ also satisfies condition 4, though we will %need to review its construction in order to clarify these points. 

Assume first that the initial data are asymptotically flat. The first four paragraphs in the proof of \cite{HKK}*{Proposition 2.1} yield an embedded 3-dimensional submanifold with boundary $(M',g,k)$ of the original data set, which is asymptotically flat and has the following additional properties. Each boundary component of $\partial M'$ is a 2-sphere MOTS or MITS, $M'$ has PSC (positive scalar curvature) topology in that it is diffeomorphic to a finite connected sum of spherical space forms, $S^1\times S^2$'s, and $\mathbb{R}^3$'s with a finite number of 3-balls removed (corresponding to horizons), and one of the ends coincides with $(M_{end},g,k)$. Consider now an outermost apparent horizon $\Sigma'\subset M$ with respect to $M_{end}$.
According to \cite{Galloway}*{Theorem 3.2}, each component of $\Sigma'$ must topologically be a 2-sphere. Moreover, the outermost condition implies that each component of $\partial M'$ is either separated from $M_{end}$ by $\Sigma'$, or has a transverse intersection with $\Sigma'$. We may then modify $\Sigma'$ to construct another outermost apparent horizon $\Sigma\subset M$ with respect to $M_{end}$ in the following way: if a component of $\Sigma'$ does not have a transverse intersection with $\partial M'$ then include this component within $\Sigma$, if a component of $\Sigma'$ has a transverse intersection with $\partial M'$ then replace it in $\Sigma$ with the relevant components of $\partial M'$. The manifold with boundary and one end produced by taking the closure of the component of $M\setminus\Sigma$ that contains $M_{end}$, will be denoted $M_0$.

We claim that $M_0$ must also have PSC topology. To see this, note that without loss of generality it may be assumed that $M_0$ is a proper subset of $M'$, by slightly pushing inwards any components of $\partial M_0$ that have nontrivial intersection with $\partial M'$. Since the outermost apparent horizon $\Sigma$ has a finite number of components, due to uniform bounds on its area and second fundamental form \cite{AEM}*{Theorem 4.6}, the desired claim is implied by the following fact. If $\mathcal{S}$ is a finite collection of embedded spheres within a 3-manifold $N$ of PSC topology, then all components of $N\setminus\mathcal{S}$ also admit PSC topology. By using induction, it will suffice to consider the case where $\mathcal{S}$ has a single component $S$. If $S$ is nonseparating in $N$ then it must represent a slice of an $S^1\times S^2$ summand in the prime decomposition of $N$. In this case, $N\setminus S$ has PSC topology since its prime decomposition is the same as that of $N$ minus the relevant $S^1 \times S^2$ summand. If $S$ is separating, then it induces a decomposition $N=N_1 \# N_2$. Moreover, due to the uniqueness of prime decompositions both $N_1$ and $N_2$ have PSC topology. To summarize, $M_0$ has one end and satisfies the same properties as $M'$, with the additional feature that $\partial M_0 =\Sigma$ is an outermost apparent horizon with respect to $M_{end}$.

%Say that $M$ has psc topology if $M$ is diffeomorphic to the compliment of finitely many disjoint embedded 3-discs in a %connected sum of spherical space forms and $S^1\times S^2$'s. We have the following general fact: if $\mathcal{S}$ is a finite %collection of embedded spheres in a $3$-manifold $(N,h)$ of psc topology, then all components of $M\setminus \mathcal{S}$ also %have psc topology. To see this, (by induction) it suffices to consider the case where $\mathcal{S}$ has a single component. If %a sphere $S$ is non-separating in $M$, it represents a slice of an $S^1\times S^2$ summand. If $S$ is separating, it induces a 
%decomposition $M=N_1\# N_2$ and, due to the uniqueness of prime decompositions, both $N_1$ and $N_2$ have psc topology.

Observe that the homology group $H_2(M_0,\partial M_0;\mathbb{Z})$ is generated by the spheres in the $S^1\times S^2$ summands, and therefore the desired property (1) will be achieved by dealing with these spheres. Since $H^1(M_0;\mathbb{Z})\cong H_2(M_0,\partial M_0;\mathbb{Z})$ is the torsion-free part of $H_1(M_0;\mathbb{Z})$, it will be sufficient to iteratively reduce the first Betti number $b_1=\mathrm{rk}H_1(M_0;\mathbb{Z})$. A procedure to accomplish this appears in the proof of
\cite{HKK}*{Proposition 2.1}, which we briefly outline here. Given a sphere in an $S^1\times S^2$ summand, consider the associated double cover $\pi:\widetilde{M}_1\to M_0$ classified by the mod 2 reduction of the cohomology class Poincar\'{e} dual to the sphere. By pulling back $g$ and $k$, the cover $\widetilde{M}_1$ becomes an initial data set with two asymptotically flat ends. Moreover, a simple calculation shows that $b_1(\widetilde{M}_1)=2b_1(M_0)-1$. By picking one of two ends of $\widetilde{M}_1$ as a reference and using the untrapped coordinate spheres in those ends together with the apparent horizon boundary as barriers for the outermost MOTS problem, one may apply \cite{topologicalcensorship}*{Theorem 4.2} (or rather a slight generalization of it to allow for nonstrict inequalities, see \cite{AnderssonMetzger}*{Section 5} or \cite{Eichmair}*{Remark 4.1}) to find an apparent horizon $\Sigma_1\subset\widetilde{M}_1$ separating the two ends. There will be two components of $\widetilde{M}_1\setminus\Sigma_1$, each containing an isometric copy of the asymptotically flat end $M_{end}$. A brief calculation, using the fact that $\Sigma_1$ consists of spheres, shows that one of the components of $\widetilde{M}_1\setminus \Sigma_1$ has first Betti number strictly less than $b_1(M_0)$. The metric completion of this component having smaller first Betti number will be denoted by $M_1$. Moreover, it may be assumed that $\partial M_1$ is an outermost apparent horizon, since if this is not initially the case then we may pass to an outermost apparent horizon while preserving all relevant properties, in the same way as passing from $M'$ to $M_0$ above.

Let $\mathcal{A}_i$ be the minimum area required to enclose $\partial M_i$, for $i=0,1$. We claim that $\mathcal{A}_0 \leq \mathcal{A}_1$. Let $\Omega\subset M_1$ be a region enclosing the end of $M_1$, so that $\partial\Omega$ is a competitor for the minimizing enclosure of $\partial M_1$. %The projection $p_1( \partial \Omega )\subset M_0$ has two types of components: those that lie in $\Sigma$ and those that don't. The union of the later type will separate the components of $\Sigma\setminus p_1(\partial \Omega)$ from $\mathcal{E}$. 
Clearly $\pi(\Omega)$ contains the end of $M_0$, so that $\partial \pi(\Omega)$ is a competitor for the minimizing enclosure of $\partial M_0$ and satisfies $|\partial\pi(\Omega)|\geq\mathcal{A}_0$.
Since $\partial\pi(\Omega)\subset \pi(\partial\Omega)$ and $\pi$ is area nonincreasing, we conclude that 
\begin{equation}
\mathcal{A}_0\leq |\partial\pi(\Omega)|\leq|\pi(\partial\Omega)|\leq|\partial\Omega|.
\end{equation}
The claim follows by taking the infimum over all such regions $\Omega$. If $b_1(M_1)=0$ then we set $(M_{\mathrm{ext}},g_{\mathrm{ext}},k_{\mathrm{ext}})=(M_1,\pi^* g,\pi^* k)$ to obtain the desired properties; note that property (4) is satisfied since $\Sigma=\partial M_0$. If $b_1(M_1)\neq 0$, then we may repeat the double cover argument until the first Betti number vanishes. The generalized exterior region then arises from the last of these covers.

It remains to describe the modifications necessary in the asymptotically hyperboloidal case. At the beginning and end of the proof of \cite{HKK}*{Proposition 2.1}, a perturbation to strict dominant energy condition \cite{dim8PMT}*{Theorem 22} was utilized. The analogous perturbation result in the asymptotically hyperboloidal setting is given by  
\cite{DahlSak}*{Theorem 5.2}, and may be used in the same manner. Next, we must justify the claim that $M'$ has PSC topology. In the asymptotically flat case this relies heavily on the Jang equation through the proof of \cite{AnderssonDahl+}*{Theorem 1.2}, and thus these arguments must be amended. In her analysis of the Jang equation on complete asymptotically hyperboloidal initial data sets $(M,g,k)$ satisfying the strict dominant energy condition, Sakovich proved \cite{Sakovich}*{Proposition 7.2} that there exists a region $\Omega\subset M$ enclosing $M_{end}$ with spherical apparent horizon boundary and which admits a complete metric $\widetilde{g}$ (a modification of the Jang metric) with asymptotically flat ends, exactly cylindrical geometry hear $\partial \Omega$, and satisfying
\begin{equation}
\int_\Omega\left((2+\epsilon)|\nabla \phi|_{\widetilde{g}}^2+R_{\widetilde{g}}\phi^2\right)dV_{\widetilde{g}}\geq\int_{\Omega}(\mu-|J|_g)\phi^2dV_{\widetilde{g}}
\end{equation}
for all $\phi\in C^\infty_c(\Omega)$ and some small $\epsilon>0$. This is exactly the outcome of Step 2 in the proof of \cite{AnderssonDahl+}*{Theorem 1.2}, and one may then carry out the remaining steps to conclude that $\Omega$ has the desired PSC topology. Lastly, in certain arguments the untrapped condition of large coordinates spheres in the asymptotic ends was utilized to construct barriers for the MOTS/MITS existence problem; for instance the existence of the apparent horizon $\Sigma_1$ above. In the asymptotically hyperboloidal setting with Wang asymptotics \eqref{1aojijhiji}, this condition remains valid as shown in \cite{Sakovich}*{Lemma 4.3}. Therefore, with these modifications the arguments presented above may be used to establish the desired result in the asymptotically hyperboloidal context.
\end{proof}

%{\color{red}[Remove:]
%\begin{proposition}
%    Suppose $M^3$ is a compact oriented manifold with boundary with trivial $H_2(M,\partial M;\mathbb{Z})$. If $W\subset M$ is a codimension-0 submanifold whose boundary consists of spheres, then $H_2(W,\partial W;\mathbb{Z})$ is trivial.
%\end{proposition}
%\begin{proof}
%    Just look at Mayer-Vietoris for $H^1$
%    \begin{equation}
%        0\sim H^1(\partial W)\longleftarrow H^1(W)\oplus H^1(M\setminus W)\longleftarrow H^1(M)\sim0
%    \end{equation}
%    and so $H^1(W)$ is trivial. Now recall $H^1(W)\sim H_2(W,\partial W)$.
%\end{proof}
%}

\section{Energy, Harmonic Functions, and the Jang Graph}
\label{s:massofjang}
\setcounter{equation}{0}
\setcounter{section}{3}

In this section we will obtain a lower bound for the total energy of an end within the original initial data set, in terms of a weighted $L^2$ Hessian involving asymptotically linear harmonic functions on the Jang deformation of a generalized exterior region. The next result identifies relevant features of the Jang data.

\begin{proposition}\label{t:modifiedjanggraph}
Let $(M,g,k)$ be an initial data set which is either asymptotically flat of order $q>1/2$ or asymptotically hyperboloidal, and satisfies the dominant energy condition. Then for each end $M_{end}$ there exists a `Jang triple' $(\overline{M},\overline{g},X)$ consisting of a complete Riemannian 3-manifold with ends $\{\mathcal{E}_l\}_{l=0}^{l_0}$, and a smooth vector field, all satisfying the following properties. 
\begin{enumerate}
\item The second relative homology is trivial, $H_2(\overline{M},\cup_{l=1}^{l_0} \mathcal{E}_l;\mathbb{Z})=0$.
\item $\mathcal{E}_0$ is asymptotically flat of order $\overline{q}>1/2$, with the third equation of \eqref{1} replaced with $R_{\overline{g}}=O(r^{-3})$ when $M$ is asymptotically hyperboloidal. Moreover, the ADM energy of this end is related to the original energy by $E(\mathcal{E}_0)=E(M_{end})$ when $M$ is asymptotically flat, and $E(\mathcal{E}_0)=2E_{\mathrm{hyp}}(M_{end})$ when $M$ is asymptotically hyperboloidal. Here $\overline{q}=q$ if $M$ is AF, and $\overline{q}=1$ if $M$ is AH.
%\item there is a metric $g_j$ on $S^2$ so that $\mathcal{E}_j$ is isometric to $((0,\infty)\times S^2,ds^2+g_k)$ for %$j=1,\dots, N$,
\item The area of any surface in $\overline{M}$ that separates $\cup_{l=1}^{l_0}\mathcal{E}_l$ from $\mathcal{E}_0$ is not less than the minimal area required to enclose the outermost apparent horizon $\Sigma\subset M$, where $\Sigma$ is associated with $M_{end}$ and given by Lemma \ref{l:exteriorregion}.
\item  $R_{\overline{g}}\in L^1(\mathcal{E}_0)$ when $M$ is asymptotically flat, and $R_{\overline{g}}+2\mathrm{div}_{\overline{g}}X\in L^1(\mathcal{E}_0)$ when $M$ is asymptotically hyperboloidal. In both cases $R_{\overline{g}}\geq 2|X|_{\overline{g}}^2-2\mathrm{div}_{\overline{g}}X$.
\item $|X|_{\overline{g}}=O(r^{-1-2q})$ on $\mathcal{E}_0$ when $M$ is asymptotically flat, $|X|_{\overline{g}}=O(r^{-2})$ on $\mathcal{E}_0$ when $M$ is asymptotically hyperboloidal and in this case $\lim_{r\to\infty}\int_{S_r}\langle X,\nu\rangle d\overline{A}=-4\pi E(\mathcal{E}_0)$ where $\nu$ is the unit outer normal to the coordinate sphere $S_r\subset\mathcal{E}_0$.
Furthermore, $|X|_{\overline{g}}$ is bounded on $\mathcal{E}_l$ for $l=1,\dots,l_0$. 
%\item The minimal area required to enclose $\partial M_{\mathrm{ext}}$ is not less than the minimal area required to enclose %$\Sigma$.
%\item $H_2({\overline{M}},\cup_{j=1}^N\mathcal{E}_j;\mathbb{Z})$ is trivial.
\end{enumerate}
\end{proposition}

\begin{remark}\label{remark111}
Note that in the asymptotically hyperboloidal case the Jang scalar curvature $R_{\overline{g}}$ is not necessarily integrable over $\mathcal{E}_0$, and therefore the ADM energy of this end may not be a geometric invariant. When referring to the ADM energy $E(\mathcal{E}_0)$ in this case, we will use the asymptotic coordinates employed in \cite{Sakovich}*{Corollary 6.11}.
\end{remark}

\begin{proof}
Assume first that $M$ is asymptotically flat. By Lemma \ref{l:exteriorregion}, there exists a generalized exterior region $(M_{\mathrm{ext}},g_{\mathrm{ext}},k_{\mathrm{ext}})$ associated with $M_{end}$ and a corresponding outermost apparent horizon $\Sigma\subset M$. Since $\partial M_{\mathrm{ext}}$ is itself an outermost apparent horizon, we may apply the techniques of  Metzger \cite{Metzger}*{Theorem 3.1} (see also \cite{HanKhuri}*{Theorem 1.1}) to find a smooth solution of Jang's equation
\begin{equation}\label{jange}
\left(g_{\mathrm{ext}}^{ij}-\frac{f^i f^j}{1+|\nabla f|^2}\right)\left(\frac{\nabla_{ij}f}{\sqrt{1+|\nabla f|^2}}-(k_{\mathrm{ext}})_{ij}\right)=0\quad\quad\text{ on }M_{\mathrm{ext}},
\end{equation}
such that 
\begin{equation}\label{ofanofinoihh}
f=O_3(r^{1-2q})\quad\text{ as }r\rightarrow\infty,\quad\quad\quad f(p)\rightarrow\pm\infty\quad\text{ as }\mathrm{dist}(p,\partial_{\pm}M_{\mathrm{ext}})\rightarrow 0,
\end{equation}
where $r$ is the radial coordinates in the asymptotically flat end, and $\partial M_{\mathrm{ext}}=\partial_{+}M_{\mathrm{ext}}\sqcup\partial_{-}M_{\mathrm{ext}}$ is decomposed into MOTS and MITS components respectively. It should be pointed out that although the hypotheses of the cited results do not precisely cover the current setting, only straightforward modifications are required to obtain the stated conclusion. Let $\overline{M}=M_{\mathrm{ext}}\setminus\partial M_{\mathrm{ext}}$, and set $\overline{g}=g+df^2$ to be the induced metric on the Jang graph $t=f(p)$ in the product 4-manifold $(M\times\mathbb{R},g+dt^2)$. The decay in \eqref{ofanofinoihh} yields an asymptotically flat end $\mathcal{E}_0 \subset\overline{M}$ of order $q$ with ADM energy $E(\mathcal{E}_0)=E(M_{end})$, so that $(2)$ holds. Moreover, the blow-up in \eqref{ofanofinoihh} may be refined to yield asymptotically cylindrical ends $\{\mathcal{E}_l\}_{l=1}^{l_0}$ associated with each component of $\partial M_{\mathrm{ext}}$. In particular, $(1)$ is satisfied in light of the homology statement of Lemma \ref{l:exteriorregion}, and $(3)$ follows from the fact that $\overline{g}\geq g$ along with Lemma \ref{l:exteriorregion} part $(4)$. Furthermore, if $g_0$ denotes the induced metric on $\partial M_{\mathrm{ext}}$ and $\hat{g}=dt^2 +g_0$ is the product metric on cylinders $\mathcal{C}^{\pm}=(0,\pm\infty)\times\partial_{\pm}M_{\mathrm{ext}}$, then by expressing the Jang surface as a graph over these cylinders and pulling the metric back \cite{Metzger}*{Theorem 4.1} (\cite{SchoenYau}*{Corollary 2}) yields
\begin{equation}\label{101}
|\hat{\nabla}^{a}(\overline{g}-\hat{g})|_{\hat{g}}=o(1),\quad\text{ as }t\rightarrow\pm\infty\text{ in }\mathcal{C}^{\pm},
\end{equation}
for any $a$ where $\hat{\nabla}$ is covariant differentiation with respect to $\hat{g}$. In the case of a strictly stable apparent horizon more precise asymptotics are available \cite{Yu}*{Theorem 2.1}, \cite{Zhao}*{Section 2}, although this will not be needed here. The scalar curvature of the Jang metric \cite{SchoenYau}*{(2.25)} is given by
\begin{equation}\label{sc1}
R_{\overline{g}}=16\pi(\mu-J(w))+|h-k|_{\overline{g}}^2 +2|X|_{\overline{g}}^2 -2\mathrm{div}_{\overline{g}}X,
\end{equation}
where
\begin{equation}
w=\frac{\nabla f}{1+|\nabla f|^2},\quad\quad\quad h=\frac{\nabla^2 f}{1+|\nabla f|^2},\quad\quad\quad X=(h-k)(w,\cdot).
\end{equation}
It follows that $(4)$ is valid, and $(5)$ holds as a consequence of \eqref{ofanofinoihh} and \cite{Metzger}*{Theorem 4.1}.

Consider now the asymptotically hyperboloidal case, which will be treated in a similar manner.  We may solve Jang's equation on the generalized exterior region with prescribed cylindrical blow-up at the outermost apparent horizon boundary as above, and with hyperboloidal asymptotics in the end given by
\begin{equation}\label{oanffoianinihh}
f=\sqrt{1+r^2}+\alpha \log r +\psi+O_3(r^{-1+\varepsilon})\quad\quad\text{ as }r\rightarrow\infty
\end{equation}
for some $\varepsilon\in(0,1)$, where $\alpha,\psi\in C^3(S^2)$ and $r$ is the radial coordinate of the end. This may be achieved by combining the methods of \cite{Metzger}*{Theorem 3.1} (\cite{HanKhuri}*{Theorem 1.1}) at the horizon, and the barriers of Sakovich \cite{Sakovich}*{Section 3} in the asymptotically hyperboloidal end. A brief outline is as follows. First, extend the initial data a small amount inside the horizon to obtain $(M'_{\mathrm{ext}},g'_{\mathrm{ext}},k'_{\mathrm{ext}})$ such that $\partial M'_{\mathrm{ext}}$ is both future and past trapped, and there is a neighborhood of each $\partial_{\pm}M'_{\mathrm{ext}}$ which is foliated by surfaces with $\theta_{\pm}<0$. These considerations allow for the construction of barriers at the inner boundary, and also help to control the blow-up region. At large coordinate spheres $S_r$ in the asymptotic end, local barriers may also be found since these surfaces are untrapped \cite{Sakovich}*{Lemma 4.3}. Thus, on the region within any large coordinate sphere $S_{r(\tau)}$, one may solve the Dirichlet problem for the capillarity regularization of the Jang equation in which $\tau f$ is inserted on the right-hand side of \eqref{jange}. Dirichlet conditions $f=\pm\frac{\delta}{\tau}$ are prescribed at $\partial_{\pm}M'_{\mathrm{ext}}$, while at the coordinate sphere $S_{r(\tau)}$ we set $f=\phi$. Here $\tau>0$ is a small parameter, $\delta>0$ is a fixed constant depending on local geometry, $r(\tau)$ is a radius depending on $\tau$, and $\phi$ is a function defined on the asymptotic end which is determined by the hyperboloidal barriers of Sakovich. A sequence $\tau_n \rightarrow 0$ is then chosen along with radii $r(\tau_n)\rightarrow\infty$ such that $\tau_n \phi$ remains uniformly bounded on $S_{r(\tau_n)}$; let $f_n$ denote the corresponding solutions. This ensures that a maximum principle argument yields uniform $C^1$ bounds for $\tau_n f_n$. As described in \cite{Sakovich}*{Section 5}, local parametric estimates then yield the control necessary to obtain subconvergence to a solution of \eqref{jange} satisfying \eqref{oanffoianinihh}, while \cite{Metzger}*{Section 3} provides the desired blow-up at the horizon $\partial M_{\mathrm{ext}}\subset M'_{\mathrm{ext}}$.

Having solved Jang, properties $(1)$ and $(3)$ follow as in the asymptotically flat case. In particular, the Jang manifold $(\overline{M},\overline{g})$ has again an asymptotically flat end \cite{Sakovich}*{Section 7} of order 1, and asymptotically cylindrical ends for each component of $\partial M_{\mathrm{ext}}$ all satisfying \eqref{101} and on which $|X|_{\overline{g}}$ remains bounded. According Proposition 7.1 and Lemma D.1 of \cite{Sakovich}, in $\mathcal{E}_0$ we have $|X|_{\overline{g}}=O(r^{-2})$ as well as the expansions
\begin{equation}
R_{\overline{g}}=\frac{2\Delta_{S^2}\psi}{r^3} +O(r^{-4+\varepsilon}),\quad \mathrm{div}_{\overline{g}}X=-\frac{\Delta_{S^2}\psi}{r^3}+O(r^{-4}),\quad \langle X,\nu\rangle=-\frac{\alpha}{r^2}+O(r^{-3+\varepsilon}),
\end{equation}
where the last equality is evaluated on coordinate spheres $S_r$ to which $\nu$ is the unit outer normal. Lastly, \cite{Sakovich}*{Corollary 6.11} yields $E(\mathcal{E}_0)=\alpha=2E_{\mathrm{hyp}}(M_{end})$. These facts, together with \eqref{sc1}, show that the remaining properties $(2)$, $(4)$, and $(5)$ are valid in the asymptotically hyperboloidal context.
\end{proof}

\subsection{Harmonic functions on the Jang graph}

We will next investigate the behavior of asymptotically linear harmonic functions on the Jang manifold. Consider the asymptotically cylindrical ends $\{\mathcal{E}_l\}_{l=1}^{l_0}$ of this manifold, and view them as being isometrically embedded within $(M\times\mathbb{R},g+dt^2)$. Denote by $\mathcal{E}_l^t$ the cross-section obtained by intersecting such an end with the $t$-level set; note that there exists $t_0>0$ sufficiently large such that for $|t|\geq t_0$ each cross-section is either empty or consists of a single sphere. For $t$ in this range let $\overline{M}_t$ be the metric completion of the component of $\overline{M}\setminus \cup_{l=1}^{l_0} \mathcal{E}_l^{\pm t}$ which contains $\mathcal{E}_0$. Since the approximation to the model cylinders is controlled by \eqref{101}, we may cap-off $\partial\overline{M}_t =\cup_{l=1}^{l_0}\mathcal{E}_l^{\pm t}$ with closed 3-balls $\{\widetilde{B}_l\}_{l=1}^{l_0}$ having uniformly bounded geometry independent of $t$, and let $(\widetilde{M}_t,\widetilde{g}_t)$ denote the resulting Riemannian manifold. Note that $\widetilde{M}_t$ is complete with one asymptotically flat end and satisfies $\widetilde{g}_t =\overline{g}$ on $\overline{M}_t$. By applying \cite{Bartnik}*{Theorem 3.1}, for each $t\geq t_0$ we obtain harmonic functions with prescribed linear asymptotics
\begin{equation}\label{hfunctions}
\Delta_{\widetilde{g}_t}u_t =0 \quad\quad\text{ on }\widetilde{M}_t,\quad\quad\quad u_t=a_i x^i +O_2(r^{1-\overline{q}})\quad\text{ as }r\rightarrow\infty,
\end{equation}
where $a_i$ are given constants and $x^i$, $\overline{q}$ are respectively coordinates, and the order, of the asymptotically flat end. In what follows, covariant differentiation with respect to $\widetilde{g}_t$ and $\overline{g}$ will be denoted by $\widetilde{\nabla}$ and $\overline{\nabla}$. 

\begin{lemma}\label{harmoniclemma}
There exists a sequence of levels $t_n \rightarrow \infty$ and corresponding harmonic functions $u_{t_n}$ on $\widetilde{M}_{t_n}$ satisfying the following properties.
\begin{enumerate}
\item  $u_{t_n}\rightarrow u$ in $C^{2,\alpha}_{loc}(\overline{M})$ for any $\alpha\in(0,1)$, where $u$ is a harmonic function on $(\overline{M},\overline{g})$ that admits the asymptotics of \eqref{hfunctions} in $\mathcal{E}_0$. 
\item $|\widetilde{\nabla} u_{t_n}|\rightarrow 0$ uniformly on the caps $\cup_{l=1}^{l_0}\widetilde{B}_l$. 
\item For each $l=1,\dots,l_0$ there exists a constant $c_l$ such that
$|\overline{\nabla}^{a}(u-c_l)|=o(1)$, $a=0,1,2$ as $|t|\rightarrow \infty$ in $\mathcal{E}_{l}$.
\end{enumerate}
\end{lemma}

\begin{remark}
If the apparent horizons over which the Jang graph blows-up are strictly stable, then a standard eigenfunction expansion of the harmonic functions along the cylindrical ends would easily establish the above result with exponential decay. However, in the current setting strict stability is not known, and thus we must utilize a completely different approach that is presented below.
\end{remark}

\begin{proof}
We begin by showing that $u_t$ admits uniform $C^{2,\alpha}$ bounds on compact subsets. To accomplish this
some notation and preliminary constructions are needed. For each radius $r$ let $\overline{M}_{t,r}$ be the metric completion of the bounded component of $\overline{M}_t \setminus S_r$ where $S_r \subset\mathcal{E}_0$ is a coordinate sphere, and denote by $\widetilde{M}_{t,r}$ the corresponding manifold with the caps $\{\widetilde{B}_l\}_{l=1}^{l_0}$ added.
A straightforward calculation \cite{HKK}*{Section 4} shows there is a sufficiently large radius $r_0$ such that
%    \begin{equation}
%        \Delta r^{-1/2}<-\frac14 r^{-2-1/2}(1-C_1r^{-q})
%    \end{equation}
%    for $r\geq1$ where $C_1$ is some constant depending on the geometry of $\mathcal{E}_0$. It follows that we may fix a radius %$r_1>10$ so that 
\begin{equation}\label{eq:r1def}
\Delta_{\overline{g}} r^{-1/2}<-\frac18 r^{-5/2}\qquad \text{ for } r\geq r_0. % r\geq r_1/10.
\end{equation}
%Next, let $\{x^l\}_{l=1}^3$ denote the asymptotically flat coordinates of $\mathcal{E}_0$. To prove the theorem, it suffices to %find a harmonic function asymptotic to a single coordinate, say $x^1$, satisfying the mass inequality and gradient decay.
Next observe that a slight generalization of \cite{Bartnik}*{Theorem 3.1} produces a solution to the Dirichlet problem
%\begin{equation}
%\begin{cases}
%\Delta v_0=0&\text{ in }M\setminus M_{{\frac{r_1}{2}},\infty}\\
%v_0=x^1+O_2(r^{1-q})&\text{ as } r\to\infty \\
%v_0=0&\text{ on }S_{\frac{r_1}{2}}
%\end{cases},
%\end{equation}
\begin{equation}
\Delta_{\overline{g}}v_0=0 \quad\text{ in }\overline{M}_t \setminus \overline{M}_{t,2r_0},\quad\quad v_0=0 \quad\text{ on }S_{2r_0},\quad\quad v_0=a_i x^i +O_2(r^{1-\overline{q}})\quad\text{ as }r\rightarrow\infty.
\end{equation}
Let $\eta$ be a radial cut-off function with $\eta=1$ outside $S_{3r_0}$ and $\eta=0$ inside $S_{2r_0}$, and set $v=\eta v_0$. 
Notice that $v$ may be extended in an obvious way to all of $\widetilde{M}_t$, and that $\Delta_{\widetilde{g}_t}v$
is supported on the annular region between $S_{2r_0}$ and $S_{3r_0}$. This function will serve as an `anchor approximate solution' for equation \eqref{hfunctions}.
    
To obtain a uniform barrier, let $\lambda>0$ be a parameter and consider the radial continuous function
\begin{equation}
\psi_{\lambda}=
\begin{cases}
\lambda r^{-1/2}&\text{ in } \widetilde{M}_t\setminus \widetilde{M}_{t,r_0}\\
\lambda r_0^{-1/2}&\text{ in }\widetilde{M}_{t,r_0}
\end{cases}.
\end{equation}
According to \eqref{eq:r1def}, the parameter $\lambda$ may be chosen sufficiently large so that $\Delta\psi_{\lambda}\leq-\Delta v$ in the weak sense. For $r\geq r_0$ and $t\geq t_0$, consider the solution of the Dirichlet problem
\begin{equation}
\Delta_{\widetilde{g}_t} w_{t,r}=-\Delta_{\widetilde{g}_t} v\quad \text{ in }\widetilde{M}_{t,r},\quad\quad\quad
w_{t,r}=0\quad \text{ on }\partial\widetilde{M}_{t,r}=S_r.
\end{equation}
Since $\psi_{\lambda}$ is a supersolution the maximum principle implies that $|w_{t,r}|<\psi_{\lambda}$, which is independent of $t$ and $r$. The harmonic functions $u_{t,r}=v+w_{t,r}$ then admit uniform pointwise bounds on compact subsets, and elliptic theory then gives uniform $C^{\mathrm{k},\alpha}$ estimates on compact subsets for these functions. A standard diagonal argument with exhausting domains produces subsequential convergence, as $r\rightarrow\infty$, to harmonic functions $u_t$ satisfying \eqref{hfunctions}. It follows that for any precompact $\Omega_t\subset \widetilde{M}_t$ of uniformly bounded geometry, there exist a constants $C_{\mathrm{k}}$ independent of $t$ such that 
\begin{equation}\label{oanfoia355nifnihh}
|u_t|_{C^{\mathrm{k},\alpha}(\Omega_t)}\leq C_\mathrm{k}. 
\end{equation}
Similarly, we may find a subsequence of levels $t_n \rightarrow\infty$ such that $u_{t_n}\rightarrow u$ in $C^{2,\alpha}_{loc}(\overline{M})$ for any $\alpha\in(0,1)$, where $u$ is a harmonic function on $(\overline{M},\overline{g})$ that admits the asymptotics of \eqref{hfunctions} in $\mathcal{E}_0$. This establishes $(1)$.

In what follows we will write $u_n$ in place of $u_{t_n}$ for convenience. 
Let $t_m < t_n$ and set $\mathcal{E}^{t_m ,t_n}_l \subset\mathcal{E}_l$ to be the annular region bounded by the cross-sections $\mathcal{E}_l^{t_m}$ and $\mathcal{E}_l^{t_n}$, which we assume is nonempty. Since $u_n$ is harmonic, we may multiply Laplace's equation by $u_n$ and integrate by parts while utilizing \eqref{oanfoia355nifnihh} to find
\begin{equation}\label{iwj09u3q09hjiohj}
\|\overline{\nabla} u_n\|_{L^2(\mathcal{E}_l^{t_m, t_n})}^2 =\int_{\partial\mathcal{E}_l^{t_m ,t_n}}u_n \overline{\nabla}_{\upsilon}u_n d\overline{A}\leq C,
\end{equation}
where here and below $C$ represents a constant independent of $m$, $n$ and $\upsilon$ is the unit outer normal to the boundary. To bound the second derivative, multiply Laplace's equation by $\Delta_{\overline{g}}u_n$ and integrate by parts 
\begin{equation}\label{eq-L^2 Hessian Bound}
\| \overline{\nabla}^2 u_n \|_{L^2(\mathcal{E}_l^{t_m, t_n})}^2=-\int_{\mathcal{E}_l^{t_m, t_n}}\mathrm{Ric}_{\overline{g}}(\overline{\nabla}u_n,\overline{\nabla}u_n)d\overline{V}-\int_{\partial\mathcal{E}_l^{t_m, t_n}}\overline{\nabla}^2 u_n(\overline{\nabla}u_n,\upsilon)d\overline{A}
\leq C,
\end{equation}
where in the last inequality we utilized \eqref{101}, \eqref{oanfoia355nifnihh}, as well as \eqref{iwj09u3q09hjiohj}. Employing similar but tedious computations, we may boot-strap to obtain the same control over all higher order derivatives. Thus by sending $n\rightarrow\infty$, Fatou's lemma implies
\begin{equation}\label{aijt9q3ju-90yuj}
\sum_{a=1}^{4}\| \overline{\nabla}^a u\|_{L^2(\mathcal{E}_l)}\leq C.
\end{equation}
By Sobolev embedding on $\mathcal{E}_l^{(t-1),t}$ we then have $|\overline{\nabla}^{a}u|=o(1)$, $a=1,2$ as $|t|\rightarrow \infty$ in $\mathcal{E}_{l}$.

To show that $u$ converges to a constant along each cylindrical end, consider a sequence of annuli $\mathcal{E}_l^{(t_n -1),t_n}$ and apply the Poincar\'{e} inequality to find constants $c_{ln}$ such that
\begin{equation}
\int_{\mathcal{E}_l^{(t_n -1),t_n}}(u-c_{ln})^2 d\overline{V}\leq C\int_{\mathcal{E}_l^{(t_n -1),t_n}}|\overline{\nabla}u|^2 d\overline{V},
\end{equation}
where $C$ is independent of $n$. The constants $c_{ln}$, which are the average value of $u$ over the annuli, are uniformly bounded by \eqref{oanfoia355nifnihh} and thus admit a convergent subsequence $c_{ln_{i}} \rightarrow c_l$. In light of \eqref{aijt9q3ju-90yuj}, Sobolev embedding shows that 
\begin{equation}
\sup_{\mathcal{E}_l^{(t_{n_i} -1),t_{n_i}}}|u-c_l|=o(1)\quad\quad\text{ as }i\rightarrow\infty. 
\end{equation}
This allows for an application of the maximum principle on the intervening annuli $\mathcal{E}_l^{t_{n_i},t_{n_{j}}}$ with $t_{n_i}<t_{n_{j}}$ yields the full limit $|u-c_l|=o(1)$ as $|t|\rightarrow\infty$. This, along with the conclusions of the previous paragraph, confirms $(3)$. 

Lastly, let $\{t(n)\}$ be a sequence with $t(n)<t_n$ which tends to infinity sufficiently slowly to ensure that $|u-u_{n}|=o(1)$ when restricted to $\mathcal{E}_l^{t(n)}$, and similarly for the difference of all derivatives. Then this sequence of functions satisfies the following boundary value problem 
\begin{equation}
\Delta_{\widetilde{g}_{t_{n}}}u_{n}=0\quad\text{ in }\mathcal{E}_{l}^{t(n), t_n} \cup \widetilde{B}_l ,\quad\quad|\overline{\nabla}^a(u_{n}-c_l)|=o(1) \quad\text{ on }\partial(\mathcal{E}_{l}^{t(n), t_n} \cup \widetilde{B}_l)=\mathcal{E}_l^{t(n)},
\end{equation}
for $a=0,1,2$ where the triangle inequality was used to obtain the boundary condition. Standard estimates for harmonic functions then yield the desired property $(2)$. 
\end{proof}

\subsection{Mass lower bounds on the Jang graph}

We will now estimate the mass of the Jang manifold $(\overline{M},\overline{g})$ produced by Proposition \ref{t:modifiedjanggraph} in terms of harmonic functions. The idea is based upon \cite{BKKS} where the Riemannian positive mass theorem was established for manifolds with minimal boundary. In the present context two technical difficulties arise, namely the Jang scalar curvature may take large negative values, and level sets of the relevant harmonic functions may have spherical components within the cylindrical ends of $\overline{M}$. Such components contribute in an undesirable way to the Euler characteristic of level sets, potentially leading to an ineffective mass bound. To overcome these problems we exploit the divergence structure present in the Jang scalar curvature, and additionally utilize a capping-off procedure for the cylindrical ends which involves a preliminary estimate before sending the capping locations to infinity. 

\begin{theorem}\label{t:massinequality}
Let $(\overline{M},\overline{g})$ be a `Jang manifold' produced by Proposition \ref{t:modifiedjanggraph}, and associated with
an initial data set $(M,g,k)$ which is either asymptotically flat of order $q>1/2$ or asymptotically hyperboloidal, and satisfies the dominant energy condition. Then there exist harmonic functions $u^1$, $u^2$, $u^3$ on $\overline{M}$ which form an asymptotically flat coordinate system on $\mathcal{E}_0$ and satisfy
\begin{equation}\label{eq:jangmassineq}
E(\mathcal{E}_0)\geq \frac{\overline{C}}{24\pi}\int_{\overline{M}}\frac{|\overline{\nabla}^2 u^i|^2}{|\overline{\nabla} u^i|}d\overline{V},
\end{equation}
for $i=1,2,3$ where $\overline{C}=1$ if $M$ is asymptotically flat and $\overline{C}=2$ if $M$ is asymptotically hyperboloidal.
\begin{comment}
Suppose $(M,g)$ is an orientable asymptotically flat $3$-manifold with ends $\{\mathcal{E}_j\}_{j=0}^N$ where $\mathcal{E}_0$ is asymptotically flat of order $q>1/2$, and $\mathcal{E}_j$ is isometric to $(S^2\times\mathbb{R}_+,g_{j}+ds^2)$ for some metric $g_{j}$ on $S^2$ for $j=1,\dots N$. Assume $H_2(M,\cup_{j=1}^N\mathcal{E}_j;\mathbb{Z})$ is trivial and $R_g\in L^1(\mathcal{E}_0)$. If there is a vector field $X$ so that
    \begin{enumerate}
        \item $|X|$ is bounded and $|X|=O(|x|^{-2q-1})=o(|x|^{-4})$ on $\mathcal{E}_0$,
        \item $R_g\geq 2|X|^2-2\mathrm{div}(X)$,
    \end{enumerate}
    then there exist harmonic functions $u^1,u^2,u^3$ on $M$ which form an asymptotically flat coordinate system on $\mathcal{E}_0$ and satisfy
    \begin{equation}\label{eq:jangmassineq}
        m(\mathcal{E}_0)\geq\frac{1}{48\pi}\int_M\frac{|\nabla^2 u^l|^2}{|\nabla u^l|}dV
    \end{equation}
    for $l=1,2,3$. Moreover, there is a constant $\lambda>0$ so that $|\nabla u^l|<e^{-\lambda s}$ for sufficiently large $s$ on the cylindrical ends.
\end{comment}
\end{theorem}

\begin{remark}
When $M$ is asymptotically hyperboloidal, as in Remark \ref{remark111} the ADM energy is calculated with respect to a specific asymptotic coordinate system due to potential lack of integrability of the Jang scalar curvature. Furthermore, note that 
the integrand of \eqref{eq:jangmassineq} is defined almost everywhere since the critical set of a harmonic function is always of measure zero (see \cite{Hardt}*{Theorem 1.1}).
\begin{comment}
When speaking of the mass of an asymptotically flat manifold, one usually imposes the condition $R_g\in L^1$ so that $m$ is independent of the choice of asymptotic coordinate chart, as originally observed by Bartnik \cite[Proposition 4.1]{Bartnik}. However, inspecting the proof of \cite[Proposition 4.1]{Bartnik}, this conclusion holds under the weaker integrability assumption appearing in the statement of Theorem \ref{t:massinequality}. 
\end{comment}
\end{remark}

\begin{proof}
We will first treat the case in which $M$ is asymptotically flat. Let $u^i$, $i=1,2,3$ be the harmonic functions on $(\overline{M},\overline{g})$ given by Lemma \ref{harmoniclemma} which asymptote to $x^i$ in the asymptotically flat end $\mathcal{E}_0$. Set $u=u^1$, and note that $U=(u^1,u^2,u^3)$ form harmonic coordinates in $\mathcal{E}_0$. For each large $r>0$ consider the coordinate cylinders $C_r \subset\mathcal{E}_0$ given by
\begin{equation}
C_r=\{U\mid |U|\leq r, \text{ }(u^2)^2+(u^3)^2=r^2 \}
\cup\{U\mid |U|=r, \text{ }(u^2)^2+(u^3)^2 \leq r^2\}.
%{}&\quad \cup\{y:(y^2)^2+(y^3)^2<r^2 \text{ and }u(y)=-r\}.
\end{equation}
By Lemma \ref{harmoniclemma} there exists a sequence of harmonic functions $u_{t_n}$ defined on the capped-off Jang manifolds $(\widetilde{M}_{t_n},\widetilde{g}_{t_n})$ and satisfying \eqref{hfunctions}, with $\vec{a}=(1,0,0)$; for convenience we will denote the functions by $u_n$ and the manifolds by $(\widetilde{M}_n,\widetilde{g}_n)$. Furthermore, let $\overline{M}_{nr}$ be the region bounded by the levels $\mathcal{E}_{l}^{\pm t_n}$, $l=1,\dots,l_0$ and the coordinate cylinder $C_r$, and let $\widetilde{M}_{nr}=\overline{M}_{nr}\cup_{l=1}^{l_0}\widetilde{B}_l$ be the corresponding capped-off manifold. According to Proposition \ref{t:modifiedjanggraph} (1), it holds that $H_2(\widetilde{M}_{n\infty};\mathbb{Z})=0$. Therefore 
\cite{BKKS}*{(6.28)} may be applied to find 
\begin{equation}\label{e:prelimBKKS}
\int_{C_r}\sum_i (\overline{g}_{ij,i}-\overline{g}_{ii,j})\nu^{j}d\overline{A}\geq\int_{\widetilde{M}_{nr}}\left(\frac{|\widetilde{\nabla}^2 u_n|^2}{|\widetilde{\nabla} u_n|}+R_{\widetilde{g}_n}|\widetilde{\nabla} u_n|\right)d\widetilde{V}+o(1), 
\end{equation}
where $\nu$ is the unit outer normal to $C_r$ and $o(1)$ concerns the limit as $r\rightarrow\infty$.

We will now estimate the right-hand side of \eqref{e:prelimBKKS}. Separating out the scalar curvature of the caps, utilizing the scalar curvature lower bound of Proposition \ref{t:modifiedjanggraph} $(4)$, and integrating by parts produces
\begin{align}\label{eq:massformulaIBP2}
\begin{split}
\int_{C_r}\sum_i (&\overline{g}_{ij,i}-\overline{g}_{ii,j})\nu^{j}d\overline{A}\\
&\geq\int_{\overline{M}_{nr}}\left(\frac{|\overline{\nabla}^2 u_n|^2}{|\overline{\nabla} u_n|}+2(|X|_{\overline{g}}^2-\mathrm{div}_{\overline{g}}X)|\overline{\nabla} u_n|\right)d\overline{V}+\sum_{l=1}^{l_0}\int_{\widetilde{B}_l}R_{\widetilde{g}_n}|\widetilde{\nabla} u_n|d\widetilde{V}\\
&\geq \int_{\overline{M}_{nr}}\left(\frac{|\overline{\nabla}^2 u_n|^2}{|\overline{\nabla} u_n|}+2|X|^2_{\overline{g}}|\overline{\nabla}u_n|+2\langle X,\overline{\nabla}|\overline{\nabla} u_n|\rangle\right) d\overline{V}\\
{}&\quad+\sum_{l=1}^{l_0}\underbrace{\int_{\widetilde{B}_l}R_{\widetilde{g}_n}|\overline{\nabla} u_n|d\overline{V}}_{\mathcal{I}^1_{ln}}
-\sum_{l=1}^{l_0}\underbrace{\int_{\mathcal{E}_l^{\pm t_{n}}}\langle X,\upsilon\rangle|\overline{\nabla}u_n| d\overline{A}}_{\mathcal{I}^2_{ln}}
-\underbrace{\int_{C_r}\langle X,\nu\rangle|\overline{\nabla}u_n| d\overline{A}}_{\mathcal{I}^3_{nr}},
\end{split}
\end{align}
where $\nu$ and $\upsilon$ are unit normals pointing out of $\overline{M}_{nr}$. Observe that the uniformly bounded geometry within the caps $\widetilde{B}_l$, together with the decay of Lemma \ref{harmoniclemma} $(2)$, implies that $\mathcal{I}^1_{ln}\rightarrow 0$ as $n\rightarrow\infty$. Similarly, the uniform bounds on $|X|_{\overline{g}}$ along the cylindrical ends given by Proposition \ref{t:modifiedjanggraph} $(5)$ yields $\mathcal{I}^2_{ln}\rightarrow 0$. Furthermore, according to the proof of Lemma \ref{harmoniclemma} the sequence of functions $|\nabla u_n|$ is bounded independent of $n$ and $r$ while Proposition \ref{t:modifiedjanggraph} $(5)$ gives $|X|_{\overline{g}}=O(r^{-1-2q})$, and therefore $\mathcal{I}^3_{nr}\rightarrow 0$ as $r\rightarrow\infty$ uniformly in $n$. Next note that Young's inequality and the refined Kato inequality (see for instance \cite{HKK}*{(3.11)}) for harmonic functions produce the following inequality almost everywhere:
\begin{equation}\label{8hr098qh8thyu}
2|\langle X,\overline{\nabla}|\overline{\nabla} u_n|\rangle|\leq\frac{1}{2}\frac{|\overline{\nabla}|\overline{\nabla}u_n||^2}{|\overline{\nabla}u_n|}
+2|X|^2_{\overline{g}}|\overline{\nabla}u_n|\leq \frac{1}{3}\frac{|\overline{\nabla}^2 u_n|^2}{|\overline{\nabla}u_n|}
+2|X|^2_{\overline{g}}|\overline{\nabla}u_n|.
\end{equation}
Altogether this yields
\begin{equation}
\int_{C_r}\sum_i (\overline{g}_{ij,i}-\overline{g}_{ii,j})\nu^{j}d\overline{A}\geq
\frac{2}{3}\int_{\overline{M}_{nr}}\frac{|\overline{\nabla}^2 u_n|^2}{|\overline{\nabla} u_n|}d\overline{V}+o(1).
\end{equation}
By first taking the limit as $r\rightarrow\infty$ and then the liminf as $n\rightarrow\infty$, the desired result for $u=u^1$ follows from \eqref{fpajfpojapojfpoqhhh} and Fatou's lemma. The corresponding inequalities for $u^2$ and $u^3$ are proved in the same manner.
 
Consider now the case in which $M$ is asymptotically hyperboloidal. In this setting one may follow the preceding arguments, although some differences occur due to the weaker fall-off present in the Jang triple. In particular, since the Jang scalar curvature is not necessarily integrable, we must specify the asymptotically flat chart that is being used to compute the ADM energy, namely that employed in \cite{Sakovich}*{Corollary 6.11}. Denote such coordinates by $x^i$, and use them to solve \eqref{hfunctions} to produce the harmonic functions $(u^1,u^2,u^3)$ given by Lemma \ref{harmoniclemma}.
Note that since $|u^i-x^i|=O_2(r^{1-\overline{q}})$ and $\overline{q}>1/2$, either set of asymptotic coordinates $x^i$ or $u^i$ may be used to compute the energy and obtain the same value. Now set $u=u^1$, let $u_n$ be the corresponding approximating sequence, and apply \eqref{eq:massformulaIBP2} as well as \eqref{8hr098qh8thyu} to obtain
\begin{align}\label{309gian9y0-ghig}
\begin{split}
\int_{C_r}\sum_i (\overline{g}_{ij,i}-\overline{g}_{ii,j})\nu^{j}d\overline{A}\geq
\frac{2}{3}\int_{\overline{M}_{nr}}\frac{|\overline{\nabla}^2 u_n|^2}{|\overline{\nabla} u_n|}d\overline{V}+\sum_{l=1}^{l_0}\left(\mathcal{I}^1_{ln}-\mathcal{I}^2_{ln}\right)
-\mathcal{I}^3_{nr}.
\end{split}
\end{align}
For the same reasons given above, it holds that $\mathcal{I}^1_{ln}-\mathcal{I}^2_{ln}\rightarrow 0$ as $n\rightarrow \infty$. However, since $|X|_{\overline{g}}=O(r^{-2})$ according to Proposition \ref{t:modifiedjanggraph} $(5)$, the boundary integral 
\begin{equation}
\mathcal{I}^3_{nr}=\int_{C_r}\langle X,\nu\rangle|\overline{\nabla}u_n| d\overline{A}=\int_{C_r}\langle X,\nu\rangle d\overline{A}+o(1)
\end{equation}
does not tend to zero. To find its contribution, recall the following expression for scalar curvature in local coordinates
\begin{equation}
R_{\overline{g}}=\partial_j \left(\overline{g}^{ij}\overline{g}^{\mathrm{k}l}(\overline{g}_{li,\mathrm{k}}-\overline{g}_{l\mathrm{k},i})\right)+Q(\partial g),
\end{equation} 
where $Q(\partial \overline{g})$ is quadratic in first derives of the metric. Then working in the harmonic coordinates $u^i$ in $\mathcal{E}_0$ we have $Q(\partial \overline{g})=O(r^{-2-2\overline{q}})=o(r^{-3})$, and therefore integrating by parts on the region $D_r$ bounded by $C_r$ and the coordinate sphere $S_{2r}$ produces
\begin{align}
\begin{split}
\int_{S_{2r}}\sum_i& (\overline{g}_{ij,i}-\overline{g}_{ii,j})\nu^{j}d\overline{A}-\int_{C_r}\sum_i (\overline{g}_{ij,i}-\overline{g}_{ii,j})\nu^{j}d\overline{A}\\
{}&=\int_{D_r}R_{\overline{g}} d\overline{V}+o(1)\\
{}&=\int_{D_r}\left(R_{\overline{g}}+2\mathrm{div}_{\overline{g}}X \right) d\overline{V}-\int_{S_{2r}}2\langle X,\nu\rangle d\overline{A}+\int_{C_r}2\langle X,\nu\rangle d\overline{A}+o(1)\\
{}&=-\int_{S_{2r}}2\langle X,\nu\rangle d\overline{A}+\int_{C_r}2\langle X,\nu\rangle d\overline{A}+o(1),
\end{split}
\end{align}
where in the last equality we have used the integrability of $R_{\overline{g}}+2\mathrm{div}_{\overline{g}}X$ provided by Proposition \ref{t:modifiedjanggraph} $(4)$. Combining these observations with \eqref{309gian9y0-ghig} gives rise to
\begin{equation}
\int_{S_{2r}}\sum_i (\overline{g}_{ij,i}-\overline{g}_{ii,j})\nu^{j}d\overline{A}+\int_{S_{2r}}2\langle X,\nu\rangle d\overline{A}\geq\frac{2}{3}\int_{\overline{M}_{nr}}\frac{|\overline{\nabla}^2 u_n|^2}{|\overline{\nabla} u_n|}d\overline{V}+o(1).
\end{equation}
Notice that the flux of $2X$ converges to $-8\pi E(\mathcal{E}_0)$ as $r\rightarrow\infty$, according to Proposition \ref{t:modifiedjanggraph} $(5)$. Hence,
by first taking the limit as $r\rightarrow\infty$ and then the liminf as $n\rightarrow\infty$, the desired result for $u=u^1$ follows from \eqref{fpajfpojapojfpoqhhh} and Fatou's lemma. The corresponding inequalities for $u^2$ and $u^3$ are proved in the same manner.
\end{proof}

\section{Dong-Song's Universal Poincar\'{e}-Faber-Szeg\"{o} Inequality}
\label{s:DSinequality}
\setcounter{equation}{0}
\setcounter{section}{4}

In \cite{Dong-Song}*{Section 3} Dong-Song established an inequality, corresponding to Theorem \ref{thm-Area Hessian Inequality} below, which is reminiscent of the classical Poincar\'{e}-Faber-Szeg\"{o} inequality \cite{PolyaSzego}
\begin{equation}\label{oinfaininghhhh}
\frac{\mathrm{Cap}(\Omega)}{n(n-2)\omega_n}\geq \left(\frac{\mathrm{Vol}(\Omega)}{\omega_n}\right)^{\frac{n-2}{n}}, 
\end{equation}
for bounded domains $\Omega\subset\mathbb{R}^n$. Here $\omega_n$ denotes the volume of the unit Euclidean $n$-ball, which up to scaling is the unique domain that saturates \eqref{oinfaininghhhh}, and the capacity of $\Omega$ is given by
\begin{equation}
\capacity(\Omega)=\inf \left \{\int_{\R^n}|\nabla f|^2 dV \Bigm|  f \in \mathrm{Lip}(\mathbb{R}^n), \text{ }\|f\|_{H^1}<\infty,\text{ }f|_{\Omega}\geq1 \right \}. 
\end{equation}
%    \begin{align}
%        K= \{f: \R^n \rightarrow \R: f \ge 0, f \in L^2(\R^n), \nabla f \in L^2(\R^n;\R^n)\}
%    \end{align}
The inequality of Dong-Song is obtained under the assumption of small ADM mass. In this section, we make the observation that their inequality is still valid without this condition. In particular, the strategy presented below follows closely that of \cite{Dong-Song} with additional expounding of details, and careful tracking of constants for eventual application to the Penrose inequality.
%The present section provides an exposition of Dong-Song's inequality so that we may observe the generality of their work, keep %careful track of the involved constants, and expound on some details.

%, keeping careful track of the constants, noticing that the main result of section $3$ of \cite{Dong-Song} does not require small mass, and stating the results in terms of general functions $u:M\rightarrow \R$. 

Throughout this section, we keep the following setup. Let $(M,g)$ be a complete Riemannian 3-manifold %without boundary 
equipped with $u^i\in C^\infty(M)$, $i= 1,2,3$, and define a  function $F\in C^{\infty}(M)$ by
\begin{align}\label{e:Fdef}
F=\sum_{i,j=1}^3 \left(\langle \nabla u^i,\nabla u^j\rangle-\delta^{ij}\right)^2.
\end{align}
For values $a<b$, we will study the sets 
\begin{align}
\Omega_{a,b}= F^{-1}([a,b]).   %, \quad 0<a<b.
\end{align}
%Then we define $U:M \rightarrow \R^3$ so that $U(x)=(u^1(x),u^2(x),u^3(x))$ and
%    \begin{align}
%       \Omega_{a,b}= F^{-1}([a,b]), \quad 0<a<b.
%    \end{align}
When $b$ is small, $\Omega_{0,b}$ consists of the points where the map $U(x)=(u^1(x),u^2(x),u^3(x))$ is almost a local isometry between $M$ and $\mathbb{R}^3$, see Lemma \ref{lem-U Local Diffeomorphism}. % of section \ref{s:Appendix}.  
In what follows we always assume that $F$ is nonconstant, and also that $F$ is proper on $F^{-1}((0,b))$ for small enough $b$, properties which are justified in the context of our main results by Lemma \ref{lem-Level Set Properties}.
%{\color{red}We will always assume that for some $\varepsilon, \ell >0$ that $F$ is proper on $F^{-1}((0,\varepsilon\ell])$.} 
According to this assumption, $F^{-1}(t)$ and $\Omega_{s,t}$ have finite area and volume for sufficiently small regular values $s<t$ of $F$. Here and throughout, we use $B^{\mathbb{R}^3}_\rho$ to denote the Euclidean ball of radius $\rho$ centered at the origin of $\mathbb{R}^3$. The following is the main result of this section.

%that $F$ is proper we know that for $s,t \in (0,\varepsilon \ell)$, $s<t$ that $F^{-1}(t)$ and $\Omega_{s,t}$ are compact and hence $\Vol(\Omega_{s,t})< \infty$ and for regular values $|F^{-1}(t)|<\infty$. 
%We note that one should establish the hypothesis that $F$ is proper by the geometry of $(M,g)$, combined with the properties of the functions $u_i$, $i \in \{1,2,3\}$ in a particular setting. For instance, in the case of Lemma \ref{lem-Level Set Properties} we use the fact that $\mathcal{E}_0$ is asymptotically flat, that all of the other ends are asymptotically cylindrical, and that the $u_i$ are harmonic in order to establish the properness of $F$. We we use these observations at the very end of the section in order to relate the area of a level set of $F$ to the mass of the asymptotically flat end $\mathcal{E}_0$.

%Now we state the main theorem of this section which will be proved at the end.

\begin{theorem}\label{thm-Area Hessian Inequality}
Let $\ell \ge 6$, $0<(\ell-1)\varepsilon<\frac{1}{16}$, $0< \eta<1$, and $\lambda>0$. Then there exists a regular value $\tau_0 \in (0,\ell \varepsilon)$ of $F$ such that
%so that $F^{-1}(\tau_0)$ is smooth and
\begin{align}\label{eq:areahessianinequality}
|F^{-1}(\tau_0)| \le  %\varepsilon \sqrt{C_{\ell,\varepsilon,\eta,\lambda}}(  144 \ell\varepsilon (1+\sqrt{\ell\varepsilon})^{\frac{3}{2}} )^2 
\mathbf{C}_{\ell,\varepsilon,\eta,\lambda}\left(\sum_{i=1}^3\int_M\frac{|\nabla^2 u^i|^2}{|\nabla u^i|}dV\right)^2,
\end{align}
where $\mathbf{C}_{\ell,\varepsilon,\eta,\lambda}$ is a constant given explicitly in \eqref{eq:finalareaineq}.
\end{theorem} 

\begin{remark}
Heuristically, \eqref{eq:areahessianinequality} is associated with the capacity-volume inequality \eqref{oinfaininghhhh}, since the volume of $\Omega_{0,\tau_0}$ is related to the area $|F^{-1}(\tau_0)|$ through the coarea formula, while the right-hand side of \eqref{eq:areahessianinequality} may be used to estimate the Dirichlet energy of $F$. Moreover, as $F$ is constant on $\partial \Omega_{0,\tau_0}$, its Dirichlet energy represents an upper bound for the capacity of $\Omega_{0,\tau_0}$.
\end{remark}

A core ingredient within the proof of this result is an isoperimetric-type inequality which bounds the area of $F$-level sets from below in terms of the volume of particular regions $\Omega_{a,b}$. To achieve this, estimates for the Jacobian determinant of $U$ 
from Lemma \ref{lem-Jacobian Estimate} will be combined with the fact that the functions $\chi(x) = \mathcal{H}^0(U^{-1}(x)\cap \Omega_{a,b})$, $0<a<b$ are of BV, as described in Lemma \ref{lem:index_is_bv}. The idea is that since $U$ is a local diffeomorphism, if the preimage of a point contains $n$ copies, then the volume and area will have to be more or less copied $n$ times as well. Hence, the coarea formula from $\R^3$ may be transferred onto $\Omega_{a,b}$ for level sets of $F$. In order to make this rigorous, the Sobolev inequality for BV functions will be applied to $\chi$. Below and throughout, we write $[a,b]_{reg}$ for the set of regular values of $F$ lying in $[a,b]$. In this section, all integrals over values of $F$ are understood to range over the regular values of $F$.

\begin{lemma}\label{lem-BV Isoperimetric Inequality}
Let $ \ell >4$, $0<(\ell-1)\varepsilon<\frac{1}{16}$, and $\lambda>0$. If $\inf_{t\in(0,(\ell-1)\varepsilon]_{reg}}|F^{-1}(t)|>0$, then there exists a regular value $t_0\leq(\ell -1)\varepsilon$ of $F$, such that %all regular values $t<(\ell -1)\varepsilon$ satisfy
\begin{align}
\Vol(\Omega_{t,t_0})^{\frac{2}{3}} \le \left(\frac{27}{16} \right)^{\frac{1}{3}}\left(\frac{1+\sqrt{(\ell-1)\varepsilon}}{1-\sqrt{(\ell-1)\varepsilon}}\right)(2+\lambda) |F^{-1}(t)|
\end{align}
for all regular values $t<t_0$. If $t_0 \leq t<(\ell-1)\varepsilon$, then the same inequality holds with $\Omega_{t,t_0}$ replaced by $\Omega_{t_0,t}$.
%when $t<t_0$. When $t>t_0$, the same inequality holds upon replacing $\Omega_{t,t_0}$ with $\Omega_{t_0,t}$.
%    and if $t\geq t_0$ then
%    \begin{align}
%        \Vol(\Omega_{t_0,t})^{\frac{2}{3}} \le \left(\frac{27}{16} \right)^{\frac{1}{3}}\left(\frac{1+\sqrt{\varepsilon}}{1-\sqrt{\varepsilon}}\right)(2+\lambda) \Area(F^{-1}(t)).
%    \end{align}
\end{lemma}

\begin{proof}
Fix a regular value $t_0\in(0,(\ell - 1)\varepsilon]$ such that $|F^{-1}(t_0)|\leq(1+\lambda)\inf_{t\in(0,(\ell-1)\varepsilon]_{reg}}|F^{-1}(t)|$. We will prove the inequality when the regular value $t$ lies in $[0,t_0)$; the case where $t \in [t_0,(\ell - 1)\varepsilon)$ follows in exactly the same way. %The proof proceeds by leveraging the Euclidean isoperimetric inequality and the fact that $g$ is $C^0$-close to the pull back of the flat metric under $U$ on $\Omega_{t,t_0}$. 
Since $|F^{-1}(t)|\geq(1+\lambda)^{-1}|F^{-1}(t_0)|$, we have%    So, we see that for the region $\Omega_{t,t_0}=F^{-1}([t,t_0])$, we have
\begin{align}
\label{eq-Boundary to Area of Level Sets}
|\partial\Omega_{t,t_0}|=|F^{-1}(t)|+|F^{-1}(t_0)|\leq(2+\lambda)|F^{-1}(t)|.
\end{align}
Write $\mathcal{J}(U)$ for the square root of the determinant of the Gram matrix $\{\langle\nabla u^i,\nabla u^j\rangle\}_{i,j=1}^3$.
Then combining \eqref{eq-Boundary to Area of Level Sets} with the estimate of this determinant in Lemma \ref{lem-Jacobian Estimate} yields 
\begin{equation}\label{eq:detpartialdu}
\int_{\partial\Omega_{t,t_0}}\mathcal{J}(U|_{\partial\Omega_{t,t_0}}) 
dA_g  \le (1+\sqrt{(\ell-1)\varepsilon}) |\partial\Omega_{t,t_0}| 
\le (1+\sqrt{(\ell-1)\varepsilon}) (2+\lambda)|F^{-1}(t)|,
\end{equation}
and
\begin{align}\label{eq:detdu}
\int_{\Omega_{t,t_0}}\mathcal{J}(U)dV_{g} &\ge \frac{4}{\sqrt{27}}(1-\sqrt{(\ell-1)\varepsilon})^{\frac{3}{2}}\Vol(\Omega_{t,t_0}).
\end{align}
Note that the hypothesis $(\ell-1)\varepsilon<\frac{1}{16}$ is used at this point with the the application of Lemma \ref{lem-Jacobian Estimate}. We will proceed by comparing the left-hand sides of \eqref{eq:detpartialdu} and \eqref{eq:detdu}.
    
Let $\mathcal{H}^0$ denote $0$-dimensional Hausdorff measure, and consider the BV function (see Lemma \ref{lem:index_is_bv})  $\chi(x) = \mathcal{H}^0(U^{-1}(x)\cap \Omega_{t,t_0})$. %, which is BV by Lemma \ref{lem:index_is_bv}. 
Recall the relation between the Jacobian matrix $\partial U=\left(\partial_l u^i \right)$ and the Gram matrix
\begin{equation}
\langle\nabla u^i ,\nabla u^j \rangle=g^{lk} \partial_l u^i \partial_k u^j =\left(\partial U \cdot g^{-1} \cdot \partial U^T  \right)^{ij},
\end{equation}
where $\cdot$ and the superscript $T$ denote matrix multiplication and transpose, respectively. It follows that the Jacobian determinant and Gram determinant satisfy
\begin{equation}
|\det \partial U|=\mathcal{J}(U) \sqrt{\det g}.
\end{equation}
Thus, we may apply the area formula using $U$ as a change of variables to find
\begin{align}
\begin{split}\int_{\Omega_{t,t_0}}\mathcal{J}(U)dV_{g}&=\int_{\mathbb{R}^3}\mathcal{H}^0(U^{-1}(x)\cap\Omega_{t,t_0})dV_{\delta}%\label{eq-Det to Counting}
\\&\leq\int_{\mathbb{R}^{3}}|\chi|^{\frac{3}{2}}dV_\delta%\label{eq-Integer Value Consequence}
\\&\leq\left(\int_{\mathbb{R}^3}|D\chi|dV_\delta\right)^{\frac32}\label{eq-Uses Euclidean Sobolev for BV}
\\&=\left(\int_{0}^{\infty}|\partial\{\chi<s\}|ds\right)^{\frac32},
\end{split}
\end{align}
where %in the first line\eqref{eq-Det to Counting} we use  the properties of determinants of differentials of maps, 
the second line follows from
%\eqref{eq-Integer Value Consequence} we use 
the fact that $\chi$ is integer valued, the third is 
%in \eqref{eq-Uses Euclidean Sobolev for BV} we apply 
the Sobolev inequality for BV functions \cite{EG}*{Theorem 5.10}, and the final equality is the coarea formula for BV functions \cite{EG}*{Theorem 5.9} involving the perimeter measure $|\partial\{\chi<s\}|$. %where we know $\chi$ is a BV function by Lemma \ref{lem:index_is_bv}.
%By the upper semicontinuity established in Lemma \ref{lem:properties_of_index_function} we find that $\partial\{\chi<s\}=\partial\{\chi\geq s\}$ \textcolor{blue}{[Why is this due to upper semicontinuity?]} \textcolor{red}{Actually this is just because $\{\chi<s\}^c=\{\chi \ge s\}$}. 
Note that $\partial\{\chi<s\}=\partial\{\chi\geq s\}$, and since $\chi$ is integer valued, for each integer $n\geq 1$ we have $\{\chi\geq s\}=\{\chi\geq n\}$ for $s\in(n-1,n]$. Hence, the last integral of \eqref{eq-Uses Euclidean Sobolev for BV} becomes 
\begin{equation}
\int_{0}^{\infty}|\partial\{\chi<s\}|ds=\sum_{n=1}^{\infty}|\partial\{\chi\geq n\}|.
\end{equation}
    %From Lemma \ref{lem:index_is_bv} and the coarea formula for BV functions \cite[Theorem 5.9]{EG} we obtain
    %\begin{equation}
    %    \int_{\R^3}|D\chi|=\int_{0}^{\infty}|\partial\{\chi<t\}|dt.\label{eq-Chi BV Coarea}
    %\end{equation}
    %We have that $\partial\{\chi<t\}=\partial\{\chi\geq t\}$. Since $\chi$ is integer-valued, we see that $\{\chi\geq t\}=\{\chi\geq n\}$ for all $t\in(n-1,n]$. Therefore, the integral above becomes $\displaystyle\sum_{n=1}|\partial\{\chi\geq n\}|$. 
Next, note that Lemma \ref{lem:properties_of_index_function} shows that %$\chi$ is upper semi-continuous and its 
the discontinuities of $\chi$ are contained in $U(\partial\Omega_{t,t_0})$. It follows that 
\begin{equation}
\partial\{\chi\geq n\}\subset U(\partial\Omega_{t,t_0})\cap\{\chi\geq n\},
\end{equation}
and so 
\begin{equation}
\sum_{n=1}^{\infty}|\partial\{\chi\geq n\}|
\leq\sum_{n=1}^{\infty}\int_{U(\partial\Omega_{t,t_0})}1_{\{\chi\geq n\}}dA_\delta
=\int_{U(\partial\Omega_{t,t_0})}\chi dA_\delta
=\int_{\partial\Omega_{t,t_0}}\mathcal{J}(U|_{\partial\Omega_{t,t_0}})dA_g. \label{eq-Chi Integral Observations}
\end{equation}
As a consequence 
\begin{equation}
\left(\int_{\Omega_{t,t_0}}\mathcal{J}(U)dV_{g}\right)^{\frac{2}{3}}\leq \int_{\partial\Omega_{t,t_0}}\mathcal{J}(U|_{\partial\Omega_{t,t_0}})dA_g,
\end{equation}
which combines with \eqref{eq:detpartialdu} and \eqref{eq:detdu} to finish the proof.
%    Now by combining \eqref{eq-Chi BV Coarea} with \eqref{eq-Chi Integral Observations} we obtain
%    \begin{align}
       % \begin{split}\int_{\R^3}|D\chi|&=\int_{0}^{\infty}|\partial\{\chi<t\}|dt 
        %\\&\leq\sum_{n=1}\int_{U(\partial\Omega_{t,t_0})}1_{\{\chi\geq n\}}\label{eq-Sum Integral Indicator Function}
        %\\&=\int_{U(\partial\Omega_{t,t_0})}\chi=\int_{\partial\Omega}\mathrm{det}_{\partial\Omega_{t,t_0}}(dU),
        %\end{split}
    %\end{align}
    %where we used that $\chi$ is integer valued in \eqref{eq-Sum Integral Indicator Function}.
%    On the other hand, Lemma \ref{lem-Jacobian Estimate} and \eqref{eq-Boundary to Area of Level Sets} imply
%    \begin{align}
%        \int_{\Omega_{t,t_0}}\det(dU)dV_{g} &\ge \frac{4}{\sqrt{27}}(1-\sqrt{\varepsilon})^{\frac{3}{2}}\Vol(\Omega_{t,t_0}),
%    \end{align}
%    and
%    \begin{align}
%        \begin{split}
%        \int_{\partial\Omega}\mathrm{det}_{\partial\Omega_{t,t_0}}(dU) dA_g & \le (1+\sqrt{\varepsilon}) |\partial\Omega_{t,t_0}| 
%        \\&\le (1+\sqrt{\varepsilon}) (2+\lambda)|F^{-1}(t)|.
%        \end{split}
%    \end{align}
%    By combining all the estimates together we obtain the desired result.
\end{proof}

%For the remainder of the section, we will always assume that $\varepsilon$ is chosen so that $t_0\pm\varepsilon$ and $\ell\varepsilon$ are regular values for $F$. 

The next goal is to relate the volume of regions $\Omega_{t_0-t,t_0}\cap\{|\nabla F|\not = 0\}$ to the Dirichlet energy of $F$. 
%To this end we will need  a capacity volume inequality the proof of which is motivated by the capacity volume inequality in $\R^n$ which is described carefully in \cite{Jeff-Capacity} and the original work \cite{PolyaSzego}. 
In order to accomplish this the following technical lemma will be utilized, which provides an inequality between the Dirichlet energies of $F$ and related radial functions $\widetilde{F}_i$, $i=1,2$ on $\R^3$. 
%This will be the key to estimating the capacity in Lemma \ref{lem-Volume Mass Estimate}.
The definition of $\widetilde{F}_i$ will depend on two cases delineated by the values of $t_0$, $\ell$, and $\varepsilon$ from Lemma \ref{lem-BV Isoperimetric Inequality}.\smallskip

\noindent{\bf{Case 1:}} $t_0\in[(\ell-3)\varepsilon,(\ell-1)\varepsilon]$. In this context set 
\begin{align}
W_1(t)=\Vol(\Omega_{t_0-t,t_0}\cap\{|\nabla F|\not = 0\}), \quad t \in [0,t_0-\varepsilon]. 
\end{align}

\noindent{\bf{Case 2:}} $t_0\in(0,(\ell-3)\varepsilon)$. In this context set 
\begin{align}
W_2(t)=\Vol(\Omega_{t_0,t_0+t}\cap\{|\nabla F|\not = 0\}), \quad t \in [0,(\ell-1)\varepsilon-t_0]. 
\end{align}
%Recalling that $\ell>4$, $W_i$ is well-defined. 
Note that $W_i$ is increasing in $t$ and $W_i(0)=0$ for $i=1,2$. We may then find radii $R_i(t)$ of the Euclidean balls with volume $W_i(t)$, that is $\omega_3 R_i(t)^3=W_i(t)$, $i=1,2$. Now declare $\widetilde{F}_i$ to be $t$ on the sphere of radius $R_i(t)$ centered at the origin of $\mathbb{R}^3$ for $i=1,2$. This defines $\widetilde{F}_1$ on $B^{\mathbb{R}^3}_{R_1(t_0-\varepsilon)}$ and $\widetilde{F}_2$ on $B^{\mathbb{R}^3}_{R_2((\ell-1)\varepsilon-t_0)}$, and we extend $\widetilde{F}_i$ to be constant outside these balls.%, resulting in Lipschitz functions on $\mathbb{R}^3$. 

\begin{lemma}\label{lem-Capacity Upper Bound Estimate}
Let $ \ell >4$, $0<(\ell-1)\varepsilon<\frac{1}{16}$, and $\lambda>0$. Assume that $\inf_{t\in(0,(\ell-1)\varepsilon]_{reg}}|F^{-1}(t)|>0$, and let $t_0\leq(\ell -1)\varepsilon$ be the regular value of $F$ given by Lemma \ref{lem-BV Isoperimetric Inequality}. 
Then $\widetilde{F}_i$ is Lipschitz for $i=1,2$, and in Case 1
\begin{align}
C_{\ell,\varepsilon, \lambda}\int_{\Omega_{\varepsilon,t_0}}|\nabla F|^2 dV_g &\ge\int_{ \mathbb{R}^3}|\nabla \widetilde{F}_1|^2 dV_{\delta},
\end{align}
and in Case 2
\begin{align}
C_{\ell,\varepsilon, \lambda} \int_{\Omega_{t_0, (\ell - 1)\varepsilon}}|\nabla F|^2 dV_g &\ge\int_{ \mathbb{R}^3}|\nabla \widetilde{F}_2|^2 dV_{\delta},
\end{align}
where 
\begin{align}
C_{\ell,\varepsilon, \lambda}=  9 \omega_3^{\frac{2}{3}}(2+\lambda)^2\left(\frac{27}{16} \right)^{\frac{2}{3}}\left(\frac{1+\sqrt{(\ell-1)\varepsilon}}{1-\sqrt{(\ell-1)\varepsilon}}\right)^2.
\end{align}
\end{lemma}

\begin{proof}
First consider Case 1. %By the coarea formula, we find for $t \in (0,t_0-\varepsilon)$
For $t \in [0,t_0-\varepsilon]$ the coarea formula implies
\begin{equation}\label{eq-Integral Over Regular Values}
W_1(t)=\Vol(\Omega_{t_0-t,t_0}\cap\{|\nabla F|\not = 0\})
=\int_{\Omega_{t_0-t, t_0}}1_{\{|\nabla F|\not  = 0\}}dV_g
=\int_{0}^t\int_{F^{-1}({t_0-s})} \frac{1}{|\nabla F|}dA_gds.
\end{equation}
It follows that $W_1$ is absolutely continuous, and thus it is differentiable almost everywhere with
\begin{align}\label{aofinaoinikhy34}
W_1'(t) = \int_{F^{-1}({t_0 -t})} \frac{1}{|\nabla F|}dA_g > 0.
\end{align}
%for almost every $t\in (0,t_0-\varepsilon)$.
%and 
%    \begin{align}
%        W'(t) = \int_{S_{t-\varepsilon}} \frac{1}{|\nabla F|}dA_g > 0,
%    \end{align}
%is well defined for almost every $t \in (0,t_0-\varepsilon)$. 
%Now for $t \in (0, t_0-\varepsilon)$, we can apply the coarea formula again to find
%    \begin{align}
%        0< W(t)=\Vol(\Omega_{\varepsilon,t_0} \cap \{|\nabla F|\not= 0\}) \le \Vol(\Omega_{\varepsilon,t_0}).
%    \end{align}
%Notice that the area of $F^{-1}(t_0-t)$ is bounded away from $0$ by assumption, and $|\nabla F|$ is clearly pointwise bounded, so $W_1'(t)>0$. 
Note that $R_1'(t)=\frac{W_1'(t)}{3\omega_3 R_1(t)^2}$ and thus the set on which it vanishes has measure zero. Furthermore, when averaged over spheres $\widetilde{F}_{i}(R_{i}(t))=t$, so that along rays $\widetilde{F}_i$ may be regarded as an inverse for $R_i$. Hence $\widetilde{F}_1$ is also absolutely continuous and therefore Lipschitz, see for instance \cite{Nat}*{(13) page 271}.

Consider now the Dirichlet energy of $F$, which by the coarea formula may be expressed as
\begin{align}\label{eq-Coarea Consequence}
\int_{\Omega_{\varepsilon,t_0}}|\nabla F|^2 dV_g = \int_{0}^{t_0-\varepsilon} \int_{F^{-1}({t_0-s})}|\nabla F|dA_gds.
\end{align}
For regular values $t_0 -s$ such that $s \in [0,t_0-\varepsilon]$, observe that H\"{o}lder's inequality implies 
\begin{align}\label{eq-Holder Consequence}
|F^{-1}(t_0-s)|^2 %= \left(\int_{F^{-1}({t_0-t})}\frac{\sqrt{|\nabla F|}}{\sqrt{|\nabla F|}}dA_g\right)^2
\le \left(\int_{F^{-1}({t_0-s})}|\nabla F|dA_g\right)\left(\int_{F^{-1}({t_0-s})}\frac{1}{|\nabla F|}dA_g\right).
\end{align}
Then combining \eqref{aofinaoinikhy34}-\eqref{eq-Holder Consequence} produces
\begin{align}\label{eq-Area and W Integral Estimate}
\int_{0}^{t_0-\varepsilon}\frac{|F^{-1}(t_0-s)|^2}{W_1'(s)}ds \le \int_{0}^{t_0-\varepsilon}\int_{F^{-1}(t_0-s)}|\nabla F|dA_gds = \int_{\Omega_{\varepsilon,t_0}}|\nabla F|^2 dV_g .
\end{align}
%The above level set areas are related to $W_1$ by Lemma \ref{lem-BV Isoperimetric Inequality}, which shows
Furthermore Lemma \ref{lem-BV Isoperimetric Inequality} yields
\begin{align}\label{eq-Power of W Estimate}
W_1(t)^{\frac{2}{3}} \le \Vol(\Omega_{t_0-t,t_0})^{\frac{2}{3}} \le (2+\lambda)\left(\frac{27}{16} \right)^{\frac{1}{3}}\left(\frac{1+\sqrt{(\ell-1)\varepsilon}}{1-\sqrt{(\ell-1)\varepsilon}}\right)|F^{-1}(t_0-t)|.
\end{align}
Thus, together with \eqref{eq-Area and W Integral Estimate} we obtain %allows us to conclude %and \eqref{eq-Power of W Estimate} we find
\begin{align}
\begin{split}
\int_0^{t_0-\varepsilon} \frac{W_1(t)^{\frac{4}{3}}}{W_1'(t)}dt & \le (2+\lambda)^2\left(\frac{27}{16} \right)^{\frac{2}{3}}\left(\frac{1+\sqrt{(\ell-1)\varepsilon}}{1-\sqrt{(\ell-1)\varepsilon}}\right)^2  \int_{0}^{t_0-\varepsilon}\frac{|F^{-1}(t_0-t)|^2}{W_1'(t)}dt %\label{eq-Ugly Integral Identity 1}
\\&\le (2+\lambda)^2\left(\frac{27}{16} \right)^{\frac{2}{3}}\left(\frac{1+\sqrt{(\ell-1)\varepsilon}}{1-\sqrt{(\ell-1)\varepsilon}}\right)^2\int_{\Omega_{\varepsilon,t_0}}|\nabla F|^2 dV_g.\label{eq-Ugly Integral Identity 2}
\end{split}
\end{align}
    
%Now we define a function which we can use to compare to the energy of $F$.  
%One should note that $R(0)=0$, $R(t_0-\varepsilon)< \infty$, $R'(t)$ is well defined almost everywhere, and $R'(t) = \frac{W'(t)}{3\omega_3 R(t)^2}$ where defined. 

The final step is to relate $W_1$ to $\widetilde{F}_1$. Note that $|\nabla \widetilde{F}_1|=\frac{1}{R_1'(t)}$ on $\partial B^{\mathbb{R}^3}_{R_1(t)}$ for almost every $t$. It follows that 
%we may continue our estimate from \eqref{eq-Ugly Integral Identity 2} to find
%We can also calculate for almost every $t \in (0,t_0-\varepsilon)$ that
%    \begin{align}
%        \widetilde{F}'(x) = \frac{1}{R'(t)}, \quad x \in \partial B_{\delta}(0,R(t)). 
%    \end{align}
%Our goal is to compare the capacity of $\widetilde{F}$ to that of $F$.
%Now by \eqref{eq-Ugly Integral Identity 2} we find
    \begin{align}
    \begin{split}
        (2+\lambda)^2\left(\frac{27}{16} \right)^{\frac{2}{3}}\left(\frac{1+\sqrt{(\ell-1)\varepsilon}}{1-\sqrt{(\ell-1)\varepsilon}}\right)^2 &\int_{\Omega_{\varepsilon,t_0}}|\nabla F|^2 dV_g 
        \\&\ge \int_{0}^{t_0-\varepsilon}\frac{W_1(t)^{\frac{4}{3}}}{W_1'(t) }dt
        \\&=\int_{0}^{t_0-\varepsilon}\frac{\omega_3^{\frac{4}{3}}R_1(t)^2}{3 \omega_3 R_1'(t) }dt
        \\&=\frac{\omega_3^{\frac{1}{3}}}{3\cdot 3 \omega_3 }\int_{0}^{t_0-\varepsilon} \int_{\partial B^{\mathbb{R}^3}_{R_1(t)}}|\nabla \widetilde{F}_1| dA_{\delta}dt\label{eq-Uses Polar Coords}
        \\&=\frac{1}{9 \omega_3^{\frac{2}{3}} }\int_{ B^{\mathbb{R}^3}_{R_1(t_0-\varepsilon)}}|\nabla \widetilde{F}_1|^2 dV_{\delta},%\label{eq-Uses Coarea Formula}                     
    \end{split}
    \end{align}
where we used polar coordinates in the second to last line and the coarea formula in the final line. Since $\widetilde{F}_1$ is constant outside $B^{\mathbb{R}^3}_{R_1(t_0-\varepsilon)}$, we obtain the desired result. A similar argument may be applied in Case 2.%, the inequality follows from the same argument.
%Hence we find the following uniform upper bound on the capacity of the ball $B_{\delta}(0,R(t))$ in Euclidean space
%    \begin{align}
%        9 \omega_3^{\frac{2}{3}}(2+\lambda)^2\left(\frac{27}{16} \right)^{\frac{2}{3}}\left(\frac{1+\sqrt{\varepsilon}}{1-\sqrt{\varepsilon}}\right)^2 \int_{\Omega_{\varepsilon,t_0}}|\nabla F|^2 dV_g &\ge\int_{ B_{\delta}(0,R(t_0-\varepsilon))}|\nabla \widetilde{F}|^2 dV_{\delta}.
%    \end{align}
%The second case follows in exactly the same way.
\end{proof}

%Now we can use Lemma \ref{lem-Capacity Upper Bound Estimate} and Lemma \ref{lem-Energy of F Estimate} to estimate the volume of an interval of regular level sets of $F$ in terms of the mass of the exterior region.
We are ready to apply the classical capacity-volume inequality \eqref{oinfaininghhhh}, to establish a lower bound for the energy of $F$ in terms of volume.

\begin{lemma}\label{lem-Volume Mass Estimate}
Let $\ell \ge 6$, $0<(\ell-1)\varepsilon<\frac{1}{16}$, $\lambda >0$, and $0 < \eta < 1$. If $\inf_{t\in(0,(\ell-1)\varepsilon]_{reg}}|F^{-1}(t)|>0$, then there exists 
%a regular value 
$0<s_0\leq(\ell-2)\varepsilon$ such that
\begin{align}
\Vol(\Omega_{s_0,s_0+\varepsilon} \cap \{|\nabla F| \not = 0\}) \le C_{\ell,\varepsilon,\lambda, \eta} \left(\int_{\Omega_{0,\ell\varepsilon}}|\nabla F|^2 dV_g\right)^3 ,
\end{align}
where
\begin{align}
C_{\ell,\varepsilon, \lambda, \eta}&=  \left(\frac{\sqrt{3}(2+\lambda)}{\min\{\ell-5-\eta,1-\frac{\eta}{2}\}}\right)^6
\left(\frac{729}{256}\right)\left(\frac{1+\sqrt{(\ell-1)\varepsilon}}{\varepsilon(1-\sqrt{(\ell-1)\varepsilon})}\right)^{6}.
\end{align}
\end{lemma}

\begin{proof}
    %Our goal now will be to define a test function which will allow us to turn Lemma \ref{lem-Capacity Upper Bound Estimate} into an upper bound on the capacity of a ball in Euclidean space. Let $t_0$ be the value appearing in the statement of Lemma \ref{lem-BV Isoperimetric Inequality}. 
Let $t_0$ be the value given by Lemma \ref{lem-BV Isoperimetric Inequality}. We proceed by relating the Dirichlet energy of $\widetilde{F}_i$ to the capacity of the Euclidean ball of radius $R_i(\varepsilon)$. To leverage \eqref{oinfaininghhhh}, however, $\widetilde{F}_i$ must be modified to a function which is at least 1 on this ball and decays to 0. This modification is carried out separately in the two cases determined by the value of $t_0$.% In order to estimate the capacity of balls in Euclidean space we will need to split into two different cases depending on the size of $t_0$.
\smallskip

\noindent\textbf{Case 1:} $t_0 \in [(\ell-3)\varepsilon, (\ell - 1) \varepsilon]$.
        %Remember the definitions
        %\begin{align}
        %   W(t)=\Vol(\Omega_{t_0-t,t_0}\cap\{|\nabla F|\not = 0\}), \quad t \in (0,t_0-\varepsilon),
        %\end{align}
        %and $\omega_3R(t)^3 = W(t)$ where $\omega_3$ is the volume of the unit ball in Euclidean space. Now we define a radially defined function $\widetilde{F}:B_{\delta}(0,R(t_0-\varepsilon))\rightarrow \R$ so that
        %\begin{align}
        %   \widetilde{F}(x) = t, \quad x \in \partial B_{\delta}(0,R(t)). 
        %\end{align}
In this case, define ${\varphi}_1\in \mathrm{Lip}(\mathbb{R}^3)$ by
\begin{align}
{\varphi}_1(x)=\frac{t_0-\varepsilon-\widetilde{F}_1(x)}{t_0-(2+\eta)\varepsilon}
%        \begin{cases}
%            \frac{t_0-\varepsilon-\widetilde{F}_1(x)}{t_0-(2+\eta)\varepsilon} & x \in B^{\mathbb{R}^3}_{R(t_0-\varepsilon)}(0)
%            \\ 0 & \text{ otherwise}%x \in \R^3 \setminus B_{\delta}(0, R(t_0-\varepsilon))
%        \end{cases}.
\end{align}
and observe that ${\varphi}_1 \ge 1$ on the ball of radius $R_{1}(\varepsilon)$ while it vanishes outside the ball of radius $R_1(t_0-\varepsilon)$. We then have
%since
% \begin{align}
%     \hat{\varphi}(x)&= \frac{t_0-\varepsilon-\widetilde{F}}{t_0-(4+\eta)\varepsilon}
%     \\&\ge \frac{t_0-\varepsilon-\varepsilon}{t_0-(4+\eta)\varepsilon}
%     \\ &\ge \max\left\{\frac{(\ell - 3)\varepsilon-\varepsilon-\varepsilon}{(\ell -3)\varepsilon-(2+\eta)\varepsilon},\frac{(\ell - 1)\varepsilon-\varepsilon-\varepsilon}{(\ell -1)\varepsilon-(2+\eta)\varepsilon}\right\} 
%     \\&= \max\left\{\frac{\ell-5}{\ell - 5 -  \eta},\frac{\ell-3}{\ell - 3 -  \eta}\right\} >1
% \end{align}
%This implies that 
\begin{align}\label{eq-Capacity Upper Bound}
\begin{split}
\capacity\left(B^{\mathbb{R}^3}_{R_1(\varepsilon)}\right) &\le \int_{\R^3}|\nabla {\varphi}_1|^2dV_{\delta} \\
&= \frac{1}{(t_0-(2+\eta)\varepsilon)^2} \int_{B^{\mathbb{R}^3}_{R_1(t_0-\varepsilon)}} |\nabla \widetilde{F}_1|^2 dV_{\delta}\\
&\leq %\frac{9 \omega_3^{\frac{2}{3}}(2+\lambda)^2}{(t_0-(2+\eta)\varepsilon)^2}\left(\frac{27}{16} \right)^{\frac{2}{3}}\left(\frac{1+\sqrt{(\ell-1)\varepsilon}}{1-\sqrt{(\ell-1)\varepsilon}}\right)^2\int_{\Omega_{\varepsilon,t_0}}|\nabla F|^2 dV_g,
\frac{C_{\ell,\varepsilon,\lambda}}{(t_0-(2+\eta)\varepsilon)^2}\int_{\Omega_{\varepsilon,t_0}}|\nabla F|^2 dV_g,
\end{split}
\end{align}
where Lemma \ref{lem-Capacity Upper Bound Estimate} was used in the last step.
%we find
%\begin{align}\label{eq-Capacity Upper Bound}
%    \capacity(B_{\delta}(0,R(\varepsilon))) \le \frac{9 \omega_3^{\frac{2}{3}}(2+\lambda)^2}{(t_0-(2+\eta)\varepsilon)^2}\left(\frac{27}{16} \right)^{\frac{2}{3}}\left(\frac{1+\sqrt{\varepsilon}}{1-\sqrt{\varepsilon}}\right)^2\int_{\Omega_{\varepsilon,t_0}}|\nabla F|^2 dV_g.
%\end{align}
By the scaling properties of capacity \cite{EG}*{Theorem 4.15}, one may compute $\capacity\left(B^{\mathbb{R}^3}_{R}\right)=R\cdot\capacity\left(B^{\mathbb{R}^3}_1\right)=R\cdot 3\omega_3$.
%\begin{align}\label{eq-Capacity Scaling}
 %   \capacity(B_{\delta}(0,R(\varepsilon)) = R(\varepsilon)\capacity(B_{\delta}(0,1))
%\end{align}
%Now since $t_0 \ge (\ell - 3) \varepsilon$ we can combine Lemma \ref{lem-Capacity Upper Bound Estimate} and \eqref{eq-Capacity Scaling} to find
%\begin{align}\label{eq-R Upper Bound}
%    R(\varepsilon) \le \frac{9 \omega_3^{\frac{2}{3}}(2+\lambda)^2}{\capacity(B_{\delta}(0,1))(l-5-\eta)^2\varepsilon^2}\left(\frac{27}{16} \right)^{\frac{2}{3}}\left(\frac{1+\sqrt{\varepsilon}}{1-\sqrt{\varepsilon}}\right)^2\int_{\Omega_{\varepsilon,t_0}}|\nabla F|^2 dV_g.
%\end{align}
Combining this observation with the fact that $\omega_3R_1(\varepsilon)^3 = W_1(\varepsilon)=\Vol(\Omega_{t_0-\varepsilon,t_0} \cap \{|\nabla F| \not = 0\})$, we find 
\begin{align}
\begin{split}
&\Vol(\Omega_{t_0-\varepsilon,t_0} \cap \{|\nabla F| \not = 0\}) \\
&\qquad\le \left(\frac{3(2+\lambda)^2}{(\ell-5-\eta)^2\varepsilon^2}\right)^3\left(\frac{27}{16} \right)^2\left(\frac{1+\sqrt{(\ell-1)\varepsilon}}{1-\sqrt{(\ell-1)\varepsilon}}\right)^{6}\left(\int_{\Omega_{\varepsilon,t_0}}|\nabla F|^2 dV_g\right)^3.
\end{split}
\end{align}
The desired conclusion follows in this case by setting $s_0=t_0-\varepsilon$, %for $s_0 \in [(\ell-2)\varepsilon,\ell \varepsilon)$, then by renaming $s_0$ as $t_0$, and using 
noting that the domain of integration on the right-hand side satisfies $\Omega_{\varepsilon,t_0} \subset \Omega_{0,\ell\varepsilon}$.%, and that $(\ell-5-\eta)\leq (1-\frac{\eta}{2}) $. 

\smallskip

\noindent{}\textbf{Case 2:} $t_0 \in (0,(\ell-3)\varepsilon)$.
%Remember the definition
%        \begin{align}
%    W(t)=\Vol(\Omega_{t_0,t_0+t}\cap\{|\nabla F|\not = 0\}), \quad t \in (0,(\ell-1)\varepsilon-t_0),
%        \end{align}
%        and  $\widetilde{F}:B_{\delta}(0,R((\ell-1)\varepsilon-t_0)) \rightarrow \R$ so that
%        \begin{align}
%           \widetilde{F}(x) = t, \quad x \in \partial B_{\delta}(0,R(t)). 
%        \end{align}
%
Now define ${\varphi}_2\in\mathrm{Lip}(\R^3)$ by
\begin{align}
{\varphi}_2(x)=
%        \begin{cases}
\frac{(\ell - 1)\varepsilon-t_0-\widetilde{F}_2(x)}{(\ell-2-\frac{\eta}{2}) \varepsilon-t_0}, 
%        & x \in B_{\delta}(0, R((\ell - 1)\varepsilon-t_0))
%                \\ 0 & x \in \R^3 \setminus B_{\delta}(0, R((\ell - 1)\varepsilon-t_0))
%            \end{cases}
\end{align}
and observe that ${\varphi}_2\geq 1$ on the ball of radius $R_2(\varepsilon)$ while it vanishes outside the ball of radius $R_2((\ell - 1)\varepsilon-t_0)$. % $x \in B_{\delta}(0,R(\varepsilon))$ that $\hat{\varphi}(x) \ge 1$.
Estimating as in Case 1 with the help of Lemma \ref{lem-Capacity Upper Bound Estimate} produces
\begin{align}
\begin{split}
\capacity\left(B^{\mathbb{R}^3}_{R_2(\varepsilon)}\right) \le \frac{C_{\ell,\varepsilon,\lambda}}{((\ell-2-\frac{\eta}{2}) \varepsilon-t_0)^2}
\int_{\Omega_{t_0, (\ell-1)\varepsilon}}|\nabla F|^2 dV_g.
\end{split}
\end{align}
Since $\omega_3R_2(\varepsilon)^3 = W_2(\varepsilon)=\Vol(\Omega_{t_0,t_0+\varepsilon} \cap \{|\nabla F| \not = 0\})$, 
%and $\capacity\left(B^{\mathbb{R}^3}_{R_2(\varepsilon)}\right)=R_2(\varepsilon)\cdot3\omega_3$, we find
we may use the capacity scaling property to find
\begin{align}
\begin{split}
    &\Vol(\Omega_{t_0,t_0+\varepsilon} \cap \{|\nabla F| \not = 0\}) \\
    &\qquad\le \left(\frac{3(2+\lambda)^2}{(1-\frac{\eta}{2})^2\varepsilon^2}\right)^3\left(\frac{27}{16} \right)^2\left(\frac{1+\sqrt{(\ell-1)\varepsilon}}{1-\sqrt{(\ell-1)\varepsilon}}\right)^{6}\left(\int_{\Omega_{t_0,(\ell-1)\varepsilon}}|\nabla F|^2 dV_g\right)^3.
    \end{split}
\end{align}
The desired conclusion now follows by setting $s_0=t_0$, and noting that $\Omega_{t_0,(\ell-1)\varepsilon} \subset \Omega_{0,\ell\varepsilon}$.
%Now by putting both cases together we find
%\begin{align}
%     \Vol(\Omega_{t_0,t_0+\varepsilon} \cap \{|\nabla F| \not = 0\}) \le C_{\varepsilon,\lambda, \eta} \left(\int_{\Omega_{0,\ell\varepsilon}}|\nabla F|^2 dV\right)^3
%      \end{align}
%      where
%\begin{align}
%          C_{\varepsilon, \lambda, \eta}&=  \left(\frac{9\omega_3(2+\lambda)^2}{\capacity(B_{\delta}(0,1))\min\left\{(l-5-\eta)^2,(1-\frac{\eta}{2})^2 \right \}}\right)^3
%          \left(\frac{729}{256}\right)\left(\frac{1+\sqrt{\varepsilon}}{\varepsilon(1-\sqrt{\varepsilon})}\right)^{6}.
%      \end{align}
%      Then one can notice that for $\ell \ge 6$ that $\min\left\{(l-5-\eta)^2,(1-\frac{\eta}{2})^2 \right \} = (1-\frac{\eta}{2})^2 $ and we know $\capacity(B_{\delta}(0,1))= 4 \pi$, $\omega_3 = \frac{4\pi}{3}$, which yields the constant in the statement of the lemma.
\end{proof}

The final inequality needed to establish the main result of this section is an elementary relationship between the Dirichlet energy of $F$, and the Hessian of the harmonic functions $u^i$.% Now we prove that the $L^2$ energy of $F$ is bounded by the $L^2$ energy of the Hessian of $u^i$ divided by the square root of the norm of the gradient of $u^i$.

\begin{lemma}\label{lem-Energy of F Estimate}
Let $\ell>0$ and $0<\ell\varepsilon<1$, then
%For $0<\varepsilon<1$, $\ell >0$, the following holds
\begin{align}
\int_{\Omega_{0,\ell \varepsilon}} |\nabla F|^2dV_g& \le 144 \ell\varepsilon (1+\sqrt{\ell\varepsilon})^{\frac{3}{2}}\sum_{i=1}^3\int_{\Omega_{0,\ell \varepsilon}} \frac{|\nabla^2 u^i|^2}{|\nabla u^i|} d V_g.
\end{align}
\end{lemma}

\begin{proof}
On $\Omega_{0,\ell \varepsilon}$ the definition of $F$ implies that
%By definition, the following holds on $\Omega_{0,\ell \varepsilon}$
\begin{align}\label{eq-Grad u Est Good Set}
\sqrt{1-\sqrt{\ell\varepsilon}} \le |\nabla u^i| \le \sqrt{1+\sqrt{\ell\varepsilon}},%, \qquad \forall x \in \Omega_{0,\ell \varepsilon}, \quad i\in \{1,2,3\}.
\end{align}
and hence
\begin{align}\label{eq-Grad of F Est Good Set}
\begin{split}
|\nabla F|& \le \sum_{i,j=1}^3 2 |\langle \nabla u^i,\nabla u^j\rangle  - \delta^{ij}| \left(|\nabla^2 u^i| |\nabla u^j|+|\nabla u^i||\nabla^2 u^j|\right)\\
%&\le 2 \sqrt{1+\sqrt{\ell\varepsilon}}\sum_{i,j=1}^3 |\langle \nabla u^i,\nabla u^j\rangle  - \delta^{ij}|\left(|\nabla^2 u^i| +|\nabla^2 %u^j|\right)\\
&\leq  4  \sqrt{1+\sqrt{\ell\varepsilon}}\sum_{i,j=1}^3 |\langle \nabla u^i,\nabla u^j\rangle  - \delta^{ij}||\nabla^2 u^i|.
\end{split}
\end{align}
Then integrating the square of \eqref{eq-Grad of F Est Good Set} and again using \eqref{eq-Grad u Est Good Set} produces
\begin{align}
\begin{split}
\int_{\Omega_{0,\ell \varepsilon}} |\nabla F|^2 d V_g &\leq 16  (1+\sqrt{\ell\varepsilon})\int_{\Omega_{0,\ell \varepsilon}} \left( \sum_{i,j=1}^3 |\langle \nabla u^i,\nabla u^j\rangle  - \delta^{ij}||\nabla^2 u^i|\right)^2 d V_g
\\&\le  144 (1+\sqrt{\ell\varepsilon})\sum_{i,j=1}^3\int_{\Omega_{0,\ell \varepsilon}}|\langle \nabla u^i,\nabla u^j\rangle  - \delta^{ij}|^2   |\nabla^2 u^i|^2 d V_g
\\&\leq 144 \ell\varepsilon (1+\sqrt{\ell\varepsilon})\sum_{i=1}^3\int_{\Omega_{0,\ell \varepsilon}}   \frac{|\nabla^2 u^i|^2}{|\nabla u^i|} |\nabla u^i|d V_g
\\&\le 144 \ell\varepsilon (1+\sqrt{\ell\varepsilon})^{\frac{3}{2}}\sum_{i=1}^3\int_{\Omega_{0,\ell \varepsilon}}   \frac{|\nabla^2 u^i|^2}{|\nabla u^i|} d V_g . \qedhere
\end{split}
\end{align}
\end{proof}

%We now give the proof of Theorem \ref{thm-Area Hessian Inequality} and Theorem \ref{thm-intro Area Energy Inequality}. 

\begin{proof}[Proof of Theorem \ref{thm-Area Hessian Inequality}]
%    First consider the case where $\nabla^2 u^i=0$ for all $i$. Since $u^i$ is non-constant, $M$ must split as a product with $\mathbb{R}$ and is therefore noncompact. On the other hand, a brief calculation shows that $F$ is constant, which contradicts our assumption. 
Since $F$ is assumed to be nonconstant, its definition implies that $|\nabla^2 u^i|$ cannot vanish identically for all $i$. Moreover, if $\inf_{t\in(0,(\ell-1)\varepsilon]_{reg}}|F^{-1}(t)|=0$ then the conclusion of the theorem is trivial, as we can immediately find a regular value $\tau_0 \in (0,(\ell-1) \varepsilon)$ such that 
\begin{align}
|F^{-1}(\tau_0)| \le C \left(\sum_{i=1}^3\int_M\frac{|\nabla^2 u^i|^2}{|\nabla u^i|}dV_g\right)^2 ,
\end{align} 
for any given $C>0$. Therefore, it may be assumed that $\inf_{t\in(0,(\ell-1)\varepsilon]_{reg}}|F^{-1}(t)|>0$. 
    
Let $s_0 \in (0,(\ell -2) \varepsilon]$ be the value given by Lemma \ref{lem-Volume Mass Estimate}. By the coarea formula, H\"{o}lder's inequality, Lemma \ref{lem-Volume Mass Estimate}, and Lemma \ref{lem-Energy of F Estimate} we find
\begin{align}\label{9ui9u9j}
\begin{split}
\int_{s_0}^{s_0+\varepsilon} |F^{-1}(t)| dt & = \int_{\Omega_{s_0,s_0+\varepsilon}} |\nabla F| dV_g
\\& \le \left(\int_{\Omega_{s_0,s_0+\varepsilon}} |\nabla F|^2 dV_g\right)^{\frac{1}{2}} 
\Vol(\Omega_{s_0,s_0+\varepsilon}\cap \{|\nabla F|\not = 0\})^{\frac{1}{2}}
\\& \le \sqrt{C_{\ell,\varepsilon,\eta,\lambda}} \left(\int_{\Omega_{s_0,s_0+\varepsilon}} |\nabla F|^2 dV_g\right)^2 %\label{eq-Almost Done}
\\&\le \sqrt{C_{\ell,\varepsilon,\eta,\lambda}}(144 \ell\varepsilon (1+\sqrt{\ell\varepsilon})^{\frac{3}{2}})^2  \left(\sum_{i=1}^3\int_M\frac{|\nabla^2 u^i|^2}{|\nabla u^i|}dV_g\right)^2.
\end{split}
\end{align}
Since there exists a regular value in $\tau_0 \in (s_0,s_0+\varepsilon)\subset (0,\ell \varepsilon)$ so that
\begin{equation}\label{9ui9u9j1}
    |F^{-1}(\tau_0)|\leq\varepsilon^{-1}\int_{s_0}^{s_0+\varepsilon}|F^{-1}(t)|dt,
\end{equation}
the desired inequality follows by combining \eqref{9ui9u9j} and \eqref{9ui9u9j1}, and setting
\begin{equation}\label{eq:finalareaineq}
    \mathbf{C}_{\ell,\varepsilon,\eta,\lambda}=\varepsilon\sqrt{C_{\ell,\varepsilon,\eta,\lambda}}(144 \ell)^2 (1+\sqrt{\ell\varepsilon})^{3}. \qedhere
\end{equation}

\begin{comment}
In particular, there exists a regular value in $\tau_0 \in (s_0,s_0+\varepsilon)\subset (0,\ell \varepsilon)$ so that 
\begin{align}\label{eq:finalareaineq}
|F^{-1}(\tau_0)| \le   \varepsilon^{-1} \sqrt{C_{\ell,\varepsilon,\eta,\lambda}}(144 \ell\varepsilon)^2 (1+\sqrt{\ell\varepsilon})^{3}  \left(\sum_{i=1}^3\int_M\frac{|\nabla^2 u^i|^2}{|\nabla u^i|}dV_g\right)^2,
\end{align}
which gives the desired inequality with $\mathbf{C}_{\ell,\varepsilon,\eta,\lambda}$ given by the constant on the right hand side.
\end{comment}
\end{proof}

\begin{remark}\label{remarkconstant}
To obtain the best constant from Theorem \ref{thm-Area Hessian Inequality}, we will show how to minimize $\mathbf{C}_{\ell,\varepsilon,\eta,\lambda}$ with appropriate choices of $\ell$, $\varepsilon$, $\eta$, and $\lambda$ within the allowable ranges of these parameters. It is straightforward to see that for such an optimization one should take $\eta,\lambda \rightarrow 0$ and then choose $\ell = 6$; this produces
\begin{align}
\begin{split}
\mathbf{C}_{6,\varepsilon,0,0}=12^{\frac{3}{2}}\cdot 864^2 \cdot \sqrt\frac{729}{256}\cdot \varepsilon^{-2}(1+\sqrt{6\varepsilon})^3 
\left(\frac{1+\sqrt{5\varepsilon}}{1-\sqrt{5\varepsilon}}\right)^{3}.
\end{split}
\end{align}
A computation shows that this function of $\varepsilon$ is decreasing on the interval $(0,\tfrac{1}{80})$, and thus by taking $\varepsilon\to\frac{1}{80}$ the following optimal constant is achieved:
\begin{align}
\mathbf{C}_{6,\frac{1}{80},0,0}
< 3.21 \times 10^{12} .
\end{align}
%The restriction $(\ell-1)\varepsilon<\frac{1}{16}$ comes from a non-sharp estimate in Lemma \ref{lem-Jacobian Estimate}.
\end{remark}

\section{Proof of the Main Theorems}
\label{s:mainproof}
\setcounter{equation}{0}
\setcounter{section}{5}

%We are now prepared to present the proof of the Penrose-like inequality for asymptotically flat initial data.

\begin{proof}[Proof of Theorems \ref{t:main} and \ref{t:hyp}] 

We will first prove the energy inequality 
\begin{equation}\label{e:energyestimate}
\mathcal{A}\leq \mathcal{C}E^2
\end{equation}
in both the asymptotically flat and asymptotically hyperboloidal settings, and will then explain how this implies the desired mass inequality in the asymptotically flat case. Let $(M,g,k)$ be an initial data set which has one of these two types of asymptotics and satisfies the dominant energy condition. Choose an end $M_{end}$ with energy $E$ (here we forgo the subscript $\mathrm{hyp}$ in the AH case), and apply Lemma \ref{l:exteriorregion} and Proposition \ref{t:modifiedjanggraph} to obtain an outermost apparent horizon $\Sigma\subset M$ having minimum enclosing area $\mathcal{A}$, as well as an associated asymptotically flat Jang manifold $(\overline{M},\overline{g})$ admitting properties $(1)$--$(5)$. We may then apply Theorem \ref{t:massinequality} to find asymptotically linear harmonic functions $\{u^i\}_{i=1}^3$ on $\overline{M}$ satisfying
\begin{equation}\label{eq:mainproofmassineq}
E\geq\frac{1}{24\pi}\int_{\overline{M}}\frac{|\overline{\nabla}^2 u^i|^2}{|\overline{\nabla} u^i|}d\overline{V}.
\end{equation}
Next, define a smooth function $F$ on $\overline{M}$ as in \eqref{e:Fdef} utilizing the $u^i$. Since $(u^1, u^2, u^2)$ forms an asymptotically flat coordinate chart in the asymptotically flat end $\mathcal{E}_0$ of $\overline{M}$, it follows that for each $\tau>0$ the sublevel set $F^{-1}([0,\tau])$ %$\Omega_{0,\tau}$ of $F$ 
contains a possibly smaller asymptotically flat end within $\mathcal{E}_0$. On the other hand, for each $i=1,2,3$ we have that $|\overline{\nabla} u^i|$ decays to zero along the cylindrical ends of $\overline{M}$ by Lemma \ref{harmoniclemma} $(3)$. Altogether, these observations imply that for any regular value $\tau<1$, $F^{-1}(\tau)$ separates the cylindrical ends from an asymptotically flat end within $\mathcal{E}_0$. By Proposition \ref{t:modifiedjanggraph} $(3)$, the area of any such $F^{-1}(\tau)$ is not less than the minimal area $\mathcal{A}$ required to enclose $\Sigma$. 
Combining Theorem \ref{thm-Area Hessian Inequality} and \eqref{eq:mainproofmassineq} shows that there is a choice of $\tau_0 \in(0,1)$ and a universal constant $C$ so that
\begin{equation}
\mathcal{A}\leq |F^{-1}(\tau_0)|_{\overline g}\leq C\left(\sum_{i=1}^3\int_{\overline{M}}\frac{|\overline{\nabla}^2 u^i|^2}{|\overline{\nabla} u^i|}d\overline{V}\right)^2\leq (3\cdot 24\pi)^2 C E^2,
\end{equation}
yielding the desired inequality. According to Remark \ref{remarkconstant}, the constant $C$ may be estimated to produce
$\mathcal{C}:= (3\cdot 24\pi)^2 C < 10^{18}$.

It remains to reduce the mass bound to an energy bound when $(M,g,k)$ is asymptotically flat. The following is a modification of the argument presented in the last remark of \cite{dim8PMT}*{Section 6.3}.
We will first make a reduction to the case in which $(M,g,k)$ is vacuum on the chosen asymptotically flat end $M_{end}$; let $q>\frac{1}{2}$ denote the order of asymptotically flat decay. According to Lemma \ref{l:exteriorregion}, it may be assumed that $(M,g,k)$ has an outermost apparent horizon boundary $\partial M =:\Sigma$ with respect to a single end. For each $\varepsilon>0$ and $p>3$, the density theorem of Lee-Lesourd-Unger \cite{initialdatadensity}*{Theorem 3.7} yields a perturbation $(g,k)\to(g',k')$, such that: $(g',k')$ is $\varepsilon$-close to $(g,k)$ in $W^{2,p}_{-q} \times W^{1,p}_{-q-1}$, each component of $\partial M$ remains an apparent horizon of the same type, the vacuum condition $\mu'=|J'|_{g'}=0$ is satisfied in the asymptotic end, and the ADM energy-momentum $(E',P')$ of $(g',k')$ is $\varepsilon$-close to that of $(g,k)$. Note that the proof of the main theorems in the paragraph above, generalizes with minor modifications to the situation here in which the initial data $(M,g',k')$ has apparent horizon boundary. Thus, we obtain the inequality $\mathcal{A}'\leq \mathcal{C}E'^2$ for the perturbed data set, where $\mathcal{A}'$ is the minimal area required to enclose an outermost apparent horizon $\Sigma'\subset(M,g',k')$. Furthermore, since $g$ is pointwise $\varepsilon$-close to $g'$ we must have $\mathcal{A}' \geq \mathcal{A} -o(1)$ as $\varepsilon \rightarrow 0$. 
%the area of the outer minimizing enclosure of $\partial M$ with respect to $g'$ is close to that of $g$. 
Consequently, it suffices to prove that the mass inequality of Theorem \ref{t:main} follows from the energy inequality for initial data which are vacuum in the designated asymptotically flat end.
%\textcolor{red}{[Make comment why not do perturbation like this for the original argument in section 3. Maybe make comment %in section 3 that this is possible. Say something in section 3 why we do not do perturbation.]}

To complete the proof, we will show that initial data $(M,g,k)$ satisfying the hypotheses of Theorem \ref{t:main} which are vacuum in the asymptotically flat end $M_{end}$ may be smoothly deformed, while preserving $\mathcal{A}$ and the dominant energy condition, to a new data set in which the resulting ADM energy agrees with the original ADM mass. Recall that the
positive mass theorem implies $E>|P|$, unless $E=|P|=0$ and $(M,g,k)$ arises from Minkowski space. Therefore, we may assume that $E>|P|$ and fix some $\theta_0\in(\tfrac{|P|}{E},1)$. Now using that $M_{end}$ is vacuum, apply the existence result of Christodoulou-{\'O}~Murchadha \cite{Boost}*{Theorem 6.1} to obtain a vacuum spacetime developement of $(M_{end},g,k)$. This spacetime is large enough so that we may deform the asymptotically flat end of $(M,g,k)$ to a boosted slice of any slope $\theta<\theta_0$. According to \cite{Chrumass1}*{page L115}, the ADM energy-momentum $(E_{\theta},P_{\theta})$ of the boosted slice is equal to the Lorentz transformation of $(E,P)$ via a $\theta$-boost. 
%Due to the asymptotics of the spacetime, coordinate spheres $\{|x|=r\}$ of this spacetime are untrapped for sufficiently large $r$. 
By deforming $M$ within this vacuum region to a boosted slice with angle $\theta=\frac{|P|}{E}$, it follows that we obtain a new asymptotically flat initial data set with energy $E_{\theta}=\frac{E-\theta |P|}{\sqrt{1-\theta^2}}=\sqrt{E^2-|P|^2}$. Moreover, since the deformation occurs in the asymptotically flat end and within a spacetime region that may be made arbitrarily close to a portion of Minkowski space, we have that $\Sigma$ remains an outermost apparent horizon and the minimal area required to enclose it is unchanged. 
%The result then follows from \eqref{e:energyestimate}.
\end{proof}

\begin{proof}[Proof of Corollary \ref{c:hyp}]
Let $M_{end}$ be an end of $(M,g)$, and consider the asymptotically hyperboloidal initial data set $(M,g,k)$ where $k=g$. As explained at the end of the proof of Lemma \ref{l:exteriorregion}, the coordinate spheres within $M_{end}$ are untrapped. Therefore, the proof of \cite{AEM}*{Theorem 4.6} shows that there exists an outermost MITS with respect to this end; note that this surface $\Sigma$ satisfies $H=2$. Observe that there cannot be a MOTS surface (these satisfy $H=-2$) lying in $M_1$, the metric completion of the component of $M\setminus\Sigma$ that contains $M_{end}$, since otherwise the MOTS may be used as a barrier to construct another MITS which violates the outermost property of $\Sigma$. Thus, $\Sigma$ is an outermost apparent horizon. We next claim that $\Sigma$ is area outerminimizing. Consider the `brane action' \cite{ACG}*{Section 2} given by  
\begin{equation}
\mathcal{B}(\Sigma')=A(\Sigma')-2V(\Sigma'),
\end{equation}
$\Sigma'\subset M_1$ is a surface homologous to $\partial M_1$ with area $A(\Sigma')$, and $V(\Sigma')$ is the volume enclosed by $\Sigma'$ and $\partial M_1$. The obstacle problem, which consists of minimizing $\mathcal{B}$ over such surfaces, has a $C^{1,1/2}$ solution by \cite{HI}*{Theorem 1.3}. This solution weakly satisfies the $H=2$ condition and by standard regularity theory it must be smooth, and thus must agree with $\partial M_1 =\Sigma$ due to the outermost property. If $\Sigma$ is not area outerminizing, then there exists a competitor surface $\Sigma'$ with less area. It follows that $\mathcal{B}(\Sigma')<\mathcal{B}(\Sigma)$, which is a contradiction. We conclude that $\Sigma$ is an outermost apparent horizon, with minimal enclosing area $\mathcal{A}=A(\Sigma)$. 

Lastly, note that $(M_1,g,k)$ satisfies properties $(1)$--$(4)$ of Lemma \ref{l:exteriorregion}, and thus may serve as the generalized exterior region for $M_{end}$. To see this, we need only establish property $(1)$ concerning the second relative homology, as all other properties are clear. Remove all closed minimal surfaces from $M$, and let $M_0$ be the metric completion of the resulting component that contains $M_{end}$. Then according to \cite{HI}*{Lemma 4.1}, $M_0$ is diffeomorphic to $\mathbb{R}^3$ minus a finite number of balls, and the boundary 2-spheres form the outermost minimal surface. By a barrier argument, $M_1$ must be a proper subset of $M_0$. Moreover, since $\partial M_1$ also consists of 2-spheres \cite{Galloway}*{Theorem 3.2}, a similar argument as the one presented in the proof of Lemma \ref{l:exteriorregion} shows that $M_1$ has the same type of simple topology as $M_0$.
We may now apply without change the remaining parts of the proof of Theorem \ref{t:hyp}, %beyond Lemma \ref{l:exteriorregion}, 
to obtain the desired result.
\end{proof}

%\begin{proof}[Proof of Theorem \ref{t:hyp}]
%The proof proceeds as in the asymptotically flat case. Proposition \ref{t:modifiedjanggraph} yeilds an asymptotically flat manifold to which we 
%\end{proof}

\appendix

\section{Companion to Section \ref{s:DSinequality}}\label{s:Appendix}

This section collects technical facts used in Sections \ref{s:DSinequality} and \ref{s:mainproof}. The following result shows that the harmonic functions on the Jang graph described in Theorem \ref{t:massinequality} satisfy the conditions required to apply Theorem \ref{thm-Area Hessian Inequality}.
% which we felt would distract from the flow of the argument.  
%We begin by establishing some simple facts about level sets of $F$ which was defined in Section \ref{s:DSinequality}.

\begin{lemma}\label{lem-Level Set Properties}
Let $(\overline{M},\overline{g})$ be the Jang manifold of Proposition \ref{t:modifiedjanggraph}, and let $u^i\in C^\infty(\overline{M})$, $i=1,2,3$ be the harmonic functions given by Theorem \ref{t:massinequality}. If $F$ is defined with these functions on the Jang manifold as in \eqref{e:Fdef}, then the following properties hold for any $\tau \in (0,1)$.
\begin{enumerate}
\item The restriction of $F$ to $F^{-1}((0,\tau))$ yields a proper map.
%\item if $0<a<b<\tau$ then $\Omega_{a,b}$ is compact in $\overline{M}$,
\item The set of regular values of $F$ is of full measure in $[0,\tau]$. 
\item If $0<a<b< \tau$ are regular values of $F$ then $\partial \Omega_{a,b}= F^{-1}(a) \cup F^{-1}(b)$.
\end{enumerate}
\end{lemma}

\begin{proof}
Let $I\subset(0,\tau)$ be compact, then the preimage $F^{-1}(I)$ is closed since $F$ is continuous. Due to the asymptotics of $u^i$ we find that $F(x) \rightarrow 0$ as $|x|\to\infty$ on the end $\mathcal{E}_0$, and thus $F^{-1}(\tau)$ must lie within some coordinate sphere of $\mathcal{E}_0$. Moreover, along the asymptotically cylindrical ends $\mathcal{E}_1,...,\mathcal{E}_{l_0}$, Lemma \ref{harmoniclemma} $(3)$ implies that $|\overline{\nabla} u^i|\rightarrow 0$, $i =1,2,3$ as $|t| \rightarrow \infty$. Since $\sqrt{1-\sqrt{\tau}} \le |\overline{\nabla} u^i|$ on $\Omega_{0,\tau}$ by definition of $F$,
    %we know that
    %\begin{align}\label{eq-Grad u Est Good Set}
    %    \sqrt{1-\sqrt{\tau}} \le |\nabla u^i| \le \sqrt{1+\sqrt{\tau}}, \qquad \forall x \in \Omega_{0, \tau}, \quad i\in \{1,2,3\},
    %\end{align}
it follows that $F^{-1}(\tau)$ cannot extend arbitrarily far down any cylindrical end. Therefore, $F^{-1}(\tau)$ is bounded and hence compact. Likewise, if $0<a<b<\tau$ then $\Omega_{a,b}$ is compact. Since $F^{-1}(I)\subset\Omega_{a,b}$ for some $a<b$ we then have that $F^{-1}(I)$ is compact, and this establishes $(1)$. Property $(2)$ follows immediately from Sard's theorem. For property $(3)$ it is enough to observe that $\partial F^{-1}([a,b])\subset F^{-1}(\partial[a,b])$ since $\overline{M}$ has no boundary, and $F^{-1}(a)\cup F^{-1}(b)=F^{-1}(\partial[a,b])\subset \partial F^{-1}([a,b])$ since $a$, $b$ are regular values.
%Finally, items (3) and (4) follows from Sard's Theorem and the fact that $\overline{M}$ as no boundary.
%    By Sard's theorem, the set of regular values of $F$ is of full measure in $[0,\tau]$ and so we we only consider regular values of $F$ where for $0<a<b<\tau$ we see that $\partial \Omega_{a,b}= F^{-1}(a) \cup F^{-1}(b)$, since $\overline{M}$ has no boundary.
\end{proof}

%Now we want to understand the range of $F$ for which we know that 
For the remainder of this section, we leave the specific setting of the Jang graph in the preceding statement and fix the general setup described above Theorem \ref{thm-Area Hessian Inequality}, where $(M,g)$ is a complete Riemannian $3$-manifold with 3 smooth functions $u^i$, $i=1,2,3$. The following is an elementary observation about the map $U:M \rightarrow \R^3$ defined by $U=(u^1,u^2,u^3)$.
%is a local diffeomorphism.

\begin{lemma}\label{lem-U Local Diffeomorphism}
If $0\le F(x)<1$ then $U$ is a local diffeomorphism near $x$.
\end{lemma}

\begin{proof}
It suffices to show that the differential $dU_x$ at $x\in M$ has no kernel. This in turn is equivalent to the set of vectors
$\{\nabla u^i(x)\}_{i=1}^3$ being linearly dependent. Let us assume by way of contradiction that this last statement is false.
Then %there are $c_j \in \R$, $j=1,2,3$ such that $\displaystyle \sum_{i=1}^3 c_j^2=1$ and
$\displaystyle\sum_{i=1}^3 c_i\nabla u^i=0$ at $x$ for some $c_i$ with $\sum_{i=1}^3 c_i^2=1$. At this point we then have
\begin{equation}
0=\sum_{i,j=1}^3c_ic_j\langle\nabla u^i,\nabla u^j\rangle
=\sum_{i=1}^3c_i^2+\sum_{i,j=1}^3c_ic_j \left(\langle\nabla u^i,\nabla u^j\rangle-\delta^{ij}\right)>1-\sqrt{F(x)},
\end{equation}             
which is a contradiction.% if $\varepsilon<1$ and hence $U$ is a local diffeomorphism.
\end{proof}

%Now we will need some estimates for the determinant of the differential of $U$ to obtain a isoperimetric inequality in Lemma \ref{lem-BV Isoperimetric Inequality}.

The next statement compares the square root of the Gram determinant $\mathcal{J}(U):=\sqrt{|\det\langle\nabla u^i ,\nabla u^j\rangle|}$ to the values of $F$.

\begin{lemma}\label{lem-Jacobian Estimate}
%    Let $(M,g)$ be a Riemannian manifold and $U:M\rightarrow \R^3$ be given by $U(x)=(u^1(x),u^2(x),u^3(x))$ for some functions $u^i:M \rightarrow \R$, $i \in \{1,2,3\}$. If $F:M \rightarrow \R$ so that 
%    \begin{align}\label{e:Fdef1}
%    F(x)=\sum_{j,k=1}^3(g(\nabla u^j,\nabla u^k)-\delta_{jk})^2,
%    \end{align} and $\Omega_{0,b}$ be the set
%    \begin{equation}
%        \Omega_{0,b}=F^{-1}(0,b),
%    \end{equation}
%    where 
If $0<b<\frac{1}{16}$ is a regular value of $F$, then the following estimates hold on $\Omega_{0,b}$ and $\partial\Omega_{0,b}$, respectively:
\begin{align}
\frac{4}{\sqrt{27}} \left(1-\sqrt{b}\right)^{\frac{3}{2}}\leq \mathcal{J}(U)\leq\left(1+\sqrt{b}\right)^{\frac{3}{2}},\qquad\qquad \mathcal{J}(U|_{\partial\Omega_{0,b}})\leq 1+\sqrt{b}     .      
\end{align}
%{\color{blue}[Need a proof of the boundary estimate.]}{\color{red}[done]}
\end{lemma}

\begin{proof} 
On $\Omega_{0,b}$ note that, $U$ is a local diffeomorphism by Lemma \ref{lem-U Local Diffeomorphism},
and by definition of $F$ the following inequalities are satisfied
\begin{align}
\sqrt{1-\sqrt{b}} \le |\nabla u^i| \le \sqrt{1+\sqrt{b}},%, \quad i\in\{1,2,3\}, \quad x \in \Omega_{0,b},
\qquad\qquad\frac{|\langle \nabla u^i,\nabla u^j\rangle |}{|\nabla u^i||\nabla u^j|} \le \frac{\sqrt{b}}{1-\sqrt{b}},%, \quad i,j\in\{1,2,3\}, \quad i \not =j, \quad x \in \Omega_{0,b}
\end{align}
for $i,j=1,2,3$. Moreover since $\langle \nabla u^i, \nabla u^j\rangle$ is symmetric, and positive definite on the region in question, we can apply the determinant-trace inequality to find
\begin{equation}
\mathcal{J}(U)%&=\sqrt{\mathrm{det}(\langle \nabla u^i, \nabla u^j\rangle)}
\le 3^{-\frac{3}{2}}\left(\mathrm{Tr}\langle \nabla u^i, \nabla u^j\rangle\right)^{\frac{3}{2}}
= 3^{-\frac{3}{2}} \left(|\nabla u^1|^2+|\nabla u^2|^2+|\nabla u^3|^2\right)^{\frac{3}{2}}
\le \left(1+\sqrt{b}\right)^{\frac{3}{2}}.
\end{equation}
Alternatively, one may observe that $F\leq b$ implies that the eigenvalues of $\langle \nabla u^i, \nabla u^j\rangle$ are bounded above by $1+\sqrt{b}$. Similarly, the same is true of the matrix $\langle \nabla_{\partial} u^i, \nabla_{\partial} u^j\rangle$, where $\nabla_{\partial}$ denotes the induced boundary gradient on $\partial\Omega_{0,b}$.
Since the eigenvalues of the restriction of a linear operator to a subspace are less than the largest eigenvalue of the original operator, the desired upper bound for $\mathcal{J}(U|_{\partial\Omega_{0,b}})$ follows.
Furthermore, to estimate the bulk determinant from below, recall that the squared volume of a parallelepiped formed by three vectors of unit length which form mutual angles $\alpha$, $\beta$, $\gamma$ is given by
\begin{equation}
\mathrm{vol}^2 =1+2\cos\alpha \cos\beta\cos\gamma -\cos^2 \alpha -\cos^2 \beta -\cos^2 \gamma.
\end{equation}
Thus, since the Gram determinant represents the squared volume of the parallelepiped formed by $\nabla u^1$, $\nabla u^2$, and $\nabla u^3$ we obtain
\begin{align}
\begin{split}
\mathcal{J}(U)^2= |\nabla u^1|^2|\nabla u^2|^2|\nabla u^3|^2
    %\\&\quad 
&\left(1+2\prod_{i<j}\frac{\langle \nabla u^i,\nabla u^j\rangle }{|\nabla u^i||\nabla u^j|} - \sum_{i=1}^3 \left(\frac{|\langle \nabla u^i,\nabla u^j\rangle |}{|\nabla u^i||\nabla u^j|}\right)^2 \right)
\\&\ge \left(1-\sqrt{b}\right)^{3}\left(1-2\left(\frac{\sqrt{b}}{1-\sqrt{b}} \right)^3-3\left(\frac{\sqrt{b}}{1-\sqrt{b}} \right)^2 \right) 
\\&\ge \frac{16}{27} \left(1-\sqrt{b}\right)^{3},
\end{split}
\end{align}
when $b<\frac{1}{16}$.
\end{proof}

Now consider the function $\chi: \R^3 \rightarrow \R$ that appears in Section \ref{s:DSinequality} and is defined by
\begin{align}
\chi(y) = \mathcal{H}^0(U^{-1}(y) \cap \Omega_{a,b}), \quad 0\leq a<b< \frac{1}{16},
\end{align}
%with the goal of proving an isoperimetric inequality. 
%To this end, we will need to apply the coarea formula and Sobolev inequality for BV functions. 
%So we establish several important properties of $\chi$ in 
where $a$ and $b$ are regular values of $F$.
The next two lemmas establish important properties of $\chi$ which are used to carry out the arguments of Lemma \ref{lem-BV Isoperimetric Inequality}.

\begin{lemma}\label{lem:properties_of_index_function}
Let $\Omega$ be an open submanifold of $M$ with smooth boundary and compact closure $\bar{\Omega}$. Let $\Psi:M \rightarrow\mathbb{R}^{3}$ be a smooth map which is a local diffeomorphism about points in $\bar{\Omega}$, and set $\Psi^{-1}_{\bar{\Omega}}(y)=\Psi^{-1}(y)\cap\bar{\Omega}$.
  %Let $\chi:\R^{n}\rightarrow\R$ be given by 
Then the function $\chi(y)=\mathcal{H}^{0}\left(\Psi^{-1}_{\bar{\Omega}}(y)\right)$ on $\mathbb{R}^3$ satisfies the following properties.
\begin{enumerate}
\item $\chi$ only takes values in the natural numbers $\N$.
\item $\chi$ is upper semicontinuous.
\item The discontinuities of $\chi$ lie entirely within $\Psi({\partial}\Omega)$.
\end{enumerate}   
\end{lemma}

\begin{proof}
% {\color{red}[how are the results of this paragraph used?]}To begin, two preliminary properties are established. We first show $\Psi(\Omega){\subset}\mathrm{int}\Psi(\overline{\Omega})$.
% For any $y_{0}{\in}\Psi(\Omega)$ choose $x_{0}{\in}\Omega$ such that $\Psi(x_{0})=y_{0}$.
%   %Since $\Psi$ is a local diffeomorphism at $x_{0}$, 
% Let $G$ be an open set containing $x_{0}$ on which $\Psi$ is a diffeomorphism.
% Since the open set $\Psi(G{\cap}\Omega)$ about $y_{0}$ is contained in $\Psi(\overline{\Omega})$, we have the claim. Next, we claim that $\partial\Psi(\overline{\Omega}){\subset}\Psi(\partial\Omega)$.
% Since $\overline{\Omega}$ is compact, we know that $\Psi(\overline{\Omega})$ is compact, and so closed.
% Hence, we see that $\partial\Psi(\overline{\Omega})\subset\Psi(\overline{\Omega})$. However, as we have just shown, we have $\Psi(\Omega)\subset\mathrm{int}\Psi(\overline{\Omega})$.
% Since, $\partial\Psi(\overline{\Omega}) \cap\mathrm{int}\Psi(\overline{\Omega}) = \emptyset$  we see that  $\partial\Psi(\overline{\Omega}) \cap\Psi(\Omega) = \emptyset$ which shows $\partial\Psi(\overline{\Omega})\subset\Psi(\partial\Omega)$.
%We will establish property $(1)$ first.
  % It is of course possible that for $y{\in}\Psi(\overline{\Omega})$ there are points in $\Psi^{-1}\{y\}$ which are not in $\overline{\Omega}$.
  % However, let us abuse notation slightly by letting  $\Psi^{-1}$ denote the correspondence which associates to a point
  % $y\in\Psi(\overline{\Omega})$ the set $\Psi^{-1}\{y\}\cap\overline{\Omega}$.
Due to the assumption that $\Psi$ is a local diffeomorphism about points in $\bar \Omega$, as a subspace $\Psi^{-1}_{\bar{\Omega}}(y)$ has the discrete topology for any $y\in\Psi(\bar{\Omega})$.
Furthermore, $\Psi^{-1}_{\bar{\Omega}}(y)$ is closed in $\bar{\Omega}$ and is therefore compact. Hence, $\chi(y)<\infty$ and this establishes $(1)$. 

\begin{comment}
Before continuing, %we follow the exposition of \cite[Chapter 17]{AliprantisBorder}, 
for a set $A{\subset}\bar{\Omega}$, define
\begin{equation}
\Psi^{u}(A)=\{y\in\Psi\left(\bar{\Omega}\right):\Psi^{-1}\{y\}{\subset}A\}.
\end{equation}
Since $\Psi$ is continuous, 
%and $\bar{\Omega}$ is compact, the map $\Psi:\bar{\Omega}{\rightarrow}\Psi\left(\bar{\Omega}\right)$ is closed.
by \cite[Lemma 17.4 (2)]{AliprantisBorder}, for any relatively open subset $A\subset\bar{\Omega}$, the set $\Psi^{u}(A)$ is a relatively open subset of $\Psi(\bar{\Omega})$.
\end{comment}

To prove $(2)$, first note that $\chi\equiv 0$ on $\mathbb{R}^3 \setminus \Psi(\bar{\Omega})$ and thus we will restrict attention to $\Psi(\bar{\Omega})$. Now let $\{y_{i}\}_{i=1}^\infty$ be a sequence of points in $\Psi(\bar{\Omega})$ converging to some $y_{0}\in\Psi(\bar{\Omega})$. We may find pairwise disjoint neighborhoods $\displaystyle \left\{V_{l}(y_{0})\right\}_{l=1}^{\chi(y_{0})}$ in $M$ which cover $\Psi^{-1}_{\bar{\Omega}}(y_{0})$, and such that $\left.\Psi\right|_{V_{l}(y_{0})}$ is a diffeomorphism onto its image. Setting
\begin{align}
W(y_{0})=\bigcap_{l=1}^{\chi(y_{0})}\Psi(V_{l}(y_{0})),   
\end{align} 
and upon shrinking $V_{l}(y_{0})$ if necessary we may assume that $\Psi(V_{l}(y_{0}))=W(y_{0})$ for all $l$. 
%Since 
%\begin{align}
%    \Psi^{u}\left(\bar{\Omega}\cap\bigsqcup_{l=1}^{\chi(y_{0})}V_{l}(y_{0})\right)
%\end{align} 
Since $\Phi(\cup_{l=1}^{\chi(y_{0})}V_{l}(y_{0}))$ is open %in $\Psi(\bar{\Omega})$ 
and contains $y_{0}$, it contains $y_i$ for all large $i$.  
%  \begin{align}
%      y_{i}{\in}\Psi^{u}\left(\overline{\Omega}\cap\bigsqcup_{l=1}^{\chi(y_{0})}G_{l}(y_{0})\right).
%  \end{align}
If $\chi(y_{i})>\chi(y_{0})$ for some such $i$, then the pigeonhole principle provides an $l_0$ such that $V_{l_0}(y_{0})$ contains two points of $\Psi^{-1}(y_{i})$, contradicting the fact that $\Psi$ is a diffeomorphism on $V_{l_0}(y_{0})$ and establishing (2).

To prove $(3)$, note as above that $\chi$ is trivially continuous on the compliment of $\Psi(\bar{\Omega})$, so that we may restrict attention to $\Psi(\bar{\Omega})$ itself. Let $y_{0}$ be any point in $\Psi(\bar{\Omega})$ such that $\Psi^{-1}_{\bar{\Omega}}(y_{0})\cap\partial\Omega=\emptyset$. We will show that $\chi$ is continuous at $y_0$. As before, there exist pairwise disjoint neighborhoods $\displaystyle \left\{V_{l}(y_{0})\right\}_{l=1}^{\chi(y_{0})}$ in $M$ which cover $\Psi^{-1}_{\bar{\Omega}}(y_{0})$, and are such that $\left.\Psi\right|_{V_{l}(y_{0})}$ is a diffeomorphism onto a common image $W(y_0)$.
  %To this end, %since we have defined %$\Psi^{-1}_{\overline{\Omega}}\{y\}=\Psi^{-1}\{y\}\cap\overline{\Omega}$
The assumption  $\Psi^{-1}_{\bar{\Omega}}(y_{0})\cap\partial\Omega=\emptyset$ implies that $\Psi^{-1}_{\bar{\Omega}}(y_{0})\subset\Omega$.
Therefore, by setting $\widetilde{V}_{l}(y_{0})=V_{l}(y_{0})\cap\Omega$ we obtain open subsets of $M$ contained in $\Omega$ such that $\Psi(\widetilde{V}_{l}(y_{0}))$ is an open neighborhood of $y_{0}$.
%Furthermore, we may assume that they share a common image $\widetilde{V}(y_{0})$.
Let $\{y_{i}\}_{i=1}^\infty$ be a sequence of points in $\Psi(\Omega)$ converging to $y_{0}$.
Notice that $y_{i}\subset W(y_0)\cap\Psi(\Omega)$ for all large $i$, so in each $\widetilde{V}_{l}(y_{0})$ we may find a point $x_i^l$ such that $\Psi(x_i^l)=y_{i}$. Furthermore, since $\widetilde{V}_{l}(y_{0})\subset\Omega$ it follows that $x_i^l\in\Omega$.
This shows that $\chi(y_{i})\geq\chi(y_{0})$ for all sufficiently large $i$.
On the other hand, property $(2)$ implies that $\chi(y_{i})\leq\chi(y_{0})$ for all large $i$. Therefore equality is achieved, showing that $\chi$ is continuous at $y_{0}$. 
%Therefore, if $\chi$ were discontinuous at $y_{0}$, we must have $\Psi^{-1}_{\bar{\Omega}}(y_{0})\cap\partial\Omega\neq\emptyset$.
\end{proof}

\begin{lemma}\label{lem:index_is_bv}
   %Let $\Omega$ be an open submanifold of $M$ with smooth boundary and compact closure, and let $\Psi:M\rightarrow\R^{n}$ be a smooth map such that $\Psi$ is a local diffeomorphism about every point in $\overline{\Omega}$.   Then the function $\chi:\Psi(\Omega)\rightarrow\R$ given by $\chi(y)=\mathcal{H}^{0}\left(\Psi^{-1}\{y\}\cap \overline{\Omega}\right)$ is $\mathrm{BV}$.
The function $\chi$ in Lemma \ref{lem:properties_of_index_function} is BV.
\end{lemma}

\begin{proof}
Recall that an $L^1(\mathbb{R}^3)$ function $f$ is BV if
\begin{align}
\mathrm{sup} \left\{\int_{\R^3}f \mathrm{div}\pmb{\phi} \;dV_{\delta} \Bigm| \pmb{\phi}\in C^{{\infty}}_{c}(\R^{3},\R^{3}), \|\pmb{\phi}\|_{L^\infty}\leq 1 \right\} < \infty.
\end{align}
As such, let $\pmb{\phi}\in C^{{\infty}}_{c}(\R^{3},\R^{3})$ be a vector field satisfying $\|\pmb{\phi}\|_{{L^\infty}}\leq1$ and compute
\begin{equation}\label{eq-Important Rewrite}
\int_{\R^{3}}\chi\mathrm{div}\pmb{\phi} \;dV_{\delta}=\int_{\R^{3}}\sum_{n=1}^{\infty}1_{\{\chi\geq n\}}\mathrm{div}\pmb{\phi} \;dV_{\delta},
\end{equation}
where $1_V$ is the indicator function for a set $V \subset \R^3$. Note that the above sum is finite since $\chi$ is bounded. By Lemma \ref{lem:properties_of_index_function} the points of discontinuity for $\chi$ are contained in $\Psi(\partial\Omega)$, so we have that $\partial\{\chi\geq n\}\subset \Psi(\partial\Omega)\cap\{\chi\geq n\}$.
This implies that $\partial\{\chi\geq n\}$ %is $2$-rectifiable with 
has finite $\mathcal{H}^{2}$-measure.
In particular, as the measure theoretic boundary of a set is always contained in its topological boundary, $\{\chi\geq n\}$ satisfies the hypotheses of \cite{EG}*{Theorem 5.23} showing that this set has finite perimeter. Hence, $1_{\{\chi\geq n\}}$ is BV according to \cite{EG}*{Definition 5.1}.
Therefore we find that
\begin{align}
\begin{split}
\int_{\R^{3}}\sum_{n=1}^{\infty}1_{\{\chi\geq n\}}\mathrm{div}\pmb{\phi} \;dV_{\delta}&=\sum_{n=1}^{\infty}\int_{\partial\{\chi\ge n\}}\langle\pmb{\phi},\nu\rangle d\mathcal{H}^{2}%\label{eq-Structure Theorem for BV}
\\
&\leq\sum_{n=1}^{\infty}\int_{\Psi(\partial\Omega)}1_{\{\chi\geq n\}}dA_{\delta} \label{eq-Properties of Omega}
\\&=\int_{\Psi(\partial\Omega)}\chi dA_{\delta} %\label{eq-Definition of Chi}
<\infty,%\\&=\int_{\partial\Omega}\mathcal{J}(\Psi) dA_{\delta}< \infty, %\label{eq-Properties of map Psi}
\end{split}
\end{align}
where $\nu$ is the measure theoretic unit outer normal, the first equality is obtained from \cite{EG}*{Theorem 5.16}, the second line follows from the fact that $|\langle \pmb{\phi}, \nu \rangle| \le |\pmb{\phi}| |\nu| \le 1$ holds almost everywhere and $\partial\{\chi\geq n\}\subset\Psi(\partial\Omega)\cap\{\chi\geq n\}$, and the third line holds by definition of $\chi$.%, and the last line is the area formula.%, and the integral is finite since $\Psi$ is smooth.
\end{proof}

\end{document}